\documentclass[11pt,showkeys]{article}
\usepackage{amsfonts,amsmath, amsthm}
\usepackage{latexsym}
\usepackage{wasysym}
\begin{document}
\title{Metric properties of  mean wiggly continua}

\author{Jacek Graczyk\\
\small{Univ. de Paris-Sud}\\
\small{Lab. de Math\'ematiques}\\
\small{91405 Orsay, France}\\
\and 
Peter  W. Jones\\
\small{ Dept. of Math.}\\
\small{Yale, New Haven}\\
\small{ CT 06520, USA}\\
\and
Nicolae Mihalache\\
\small{Univ. Paris-Est Cr\'eteil}\\
\small{LAMA}\\
\small{ 94 010 Cr\'eteil, France}}
\maketitle
\newtheorem{lem}{Lemma}[section]
\newtheorem{conjec}{Conjecture}
\newtheorem{remark}{Remark}
\newtheorem{theo}{Theorem}
\newtheorem{coro}{Corollary}[section]
\newtheorem{prop}{Proposition}
\newtheorem{con}{Construction}[section]
\newtheorem{defi}{Definition}[section]
\newcommand{\A}{{\cal A}}
\newcommand{\B}{{\cal B}}
\newcommand{\de}{{\bf \delta}}
\newcommand{\sr}{\mbox{{{\bf I\hskip -2pt R}}}}
\newcommand{\un}{\underline}
\newtheorem{fact}{Fact}[section]
%\newenvironment{proof}
%{{\bf Proof:}\newline}{\begin{flushright}$\Box$\end{flushright}}
%\font\mathfonta=msam10 at 11pt
%\font\mathfontb=msbm10 at 11pt

\newcommand\requ[1]{(\ref{equ#1})}
\newcommand\rcor[1]{Corollary \ref{coro#1}}

\def\Z{\Bbb Z}
\def\C{\Bbb C}
\def\NN{\Bbb N}
\def\CC{\hat{\Bbb C}}
\def\RR{\Bbb R}
\def\F{{\mathcal F}}
\def\G{{\mathcal G}}
\def\B{{\mathcal B}}
\def\ss{\subset}
\def\se{\subseteq}
\def\sm{\setminus}
\def\pa{\partial}
\def\ol{\overline}

%\newtheorem{nona}{}
%\newtheorem{lem}{Lemma}[section]
%\newtheorem{conjec}{Conjecture}
%\newtheorem{theo}{Theorem}
%\newtheorem{quest}{Question}
%\newtheorem{fact}{Fact}[section]
%\newtheorem{coro}{Corollary}[section]
%\newtheorem{rem}{Remark}[section]
%newtheorem{com}{Comment}
%\newtheorem{prop}{Proposition}[section]
%\newtheorem{prob}{Problem}
%\newtheorem{con}{Construction}[section]
%\newtheorem{defi}{Definition}[section]
%\newenvironment{proof}
%{{\bf Proof:}\noindent}{\begin{flushright}$\Box$\end{flushright}}
%\newenvironment{proof}
%{{\noindent\bf Proof:~}}{\hfill$\Box$}
%%%% AMS TeX symbols
%\def\mathrm{\rm}
\font\mathfonta=msam10 at 11pt
\font\mathfontb=msbm10 at 11pt
\def\Bbb#1{\mbox{\mathfontb #1}}
\def\lesssim{\mbox{\mathfonta.}}
\def\suppset{\mbox{\mathfonta{c}}}
\def\subbset{\mbox{\mathfonta{b}}}
\def\grtsim{\mbox{\mathfonta\&}}
\def\gtrsim{\mbox{\mathfonta\&}}

%%%% Operators & functions
\def\Sig{{\bf\Sigma}}
\def\area{\mathrm{area}}
\def\diam{{\mathrm{diam}}}
\def\Jac{{\mathrm{Jac}}}
\def\Comp{{\mathrm{Comp\;}}}
\def\mod{{\rm{mod}}}
\def\spt{{\mathrm{supp\,}}}
\def\Var#1{{{\mathrm Var}\br{#1}}}
\def\dist#1{{{\mathrm{dist}}\br{#1}}}
\def\qhdist#1{{{\mathrm{dist_{qh}}}\br{#1}}}
\def\qhl#1{{{\mathrm{l_{qhyp}}}\br{#1}}}

%%%% Letters & symbols
\def\eps{\varepsilon}
\def\del{\Delta}
\def\la{\lambda}
\def\ga{\gamma}
\def\Ga{\Gamma}
\def\be{\beta}
\def\om{\Omega}
\def\al{\alpha}
\def\lam{\lambda}
\def\ph{\varphi}
\def\dd{\partial}

%%%% Sets & spaces
\def\Cal#1{{\cal#1}}
\def\J{\mathcal{J}}
\def\C{\Bbb C}
\def\N{\Bbb N}
\def\CC{\hat{\Bbb C}}
\def\D{\Bbb D}
\def\R{\Bbb R}
\def\one{\Bbb {2}}
\def\Crit{{\mathrm{Crit}}}
\def\ha{{\cal H}}
\def\J{{\cal J}}
%%%% Braces
\def\scal#1#2{\left\langle{#1},{#2}\right\rangle}
\def\br#1{\left(#1\right)}
\def\brr#1{\left[#1\right]}
\def\brs#1{\left\{#1\right\}}
\def\abs#1{\left|#1\right|}
\def\norm#1{\left\|#1\right\|}
\def\len#1{{{\mathrm{length}}\br{#1}}}
\def\qhlen#1{{{\mathrm{length_{qh}}}\br{#1}}}

%%%% Dimensions
\def\HD{{\mathrm{dim_H}}}
\def\HP{{\mathrm{dim_P}}}
\def\HDo{{\mathrm{dim_\omega}}}
\def\H{{\mathcal{H}^1}}
\def\MD{{\mathrm{MDim}}}
\def\MDsup{{\overline{\mathrm{MDim}}}}
\def\MDinf{{\underline{\mathrm{MDim}}}}
\def\HH{{\mathrm{HypDim}}}
\def\dwhit{{\delta_{Whit}}}
\def\dpoin{{\delta_{Poin}}}
\def\dconf{{\delta_{conf}}}

\def\ra{{\rightarrow}}

%%%% Etc
\def\const{{\mathrm{const}}}
\def\fr{\frac}
\def\itemm#1{ \item{\makebox[0.3in][l]{#1}}}
%\def\kerr{{\mathrm{ker}}}
%\def\eon{{e^{o(n)}}}
%\def\eonasymp{\overset{\eon}\to\asymp}

%%%% Paper secific
\def\m#1{\mu(#1)}
\def\mum{\mu}
\def\mmax{\mu_{max}}
\def\l2t{L}
\def\supF{{\sup|F'|}}
\def\xx#1#2{{{\cal H}_{#1}(#2)}}
\def\diam{{\mathrm{diam\;}}}
\def\degg#1#2{{\# \br{#1,#2}}}
\def\Jls{{J_*}}
\def\Jlso{{J_{*,\epsilon}}}
\def\typei{{\bf I}}
\def\typeii{{\bf II}}
\def\typeiii{{\bf III}}
\def\gr{{\mathrm g}}
\def\sg{{\mathrm sg}}
\def\es{{\emptyset}}
\def\rg{{\mathrm rg}}
\def\Fe{{F^{-1}_{\mathrm{rl}}}}

%newcounter{bean}
%\newcounter{dwa}
%\newcounter{jed}

\begin{abstract}
We study lower and upper bounds of the Hausdorff dimension for sets which are wiggly at scales of positive density. The main technical ingredient is a construction, for every continuum $K$, of a Borel probabilistic measure $\mu$ with the property that  on every ball $B(x,r)$, $x \in K$,   the measure is bounded by a universal constant  multiple of $r\exp(-g(x,r))$, where $g(x,r) \geq 0$ is an explicit function. The continuum $K$ is mean wiggly at exactly those points $x\in K$ where $g(x, r)$ has  a logarithmic growth to $\infty$ as $r \ra 0$. The theory
of mean wiggly continua leads, via the product formula for dimensions,  to new estimates of  the Hausdorff dimension for Cantor  sets.  We prove also that asymptotically flat sets are of Hausdorff dimension $1$ and that asymptotically non-porous continua are of the maximal dimension. Another  application of the theory is geometric Bowen's dichotomy  for Topological Collet-Eckmann maps in rational dynamics. In particular, mean wiggly continua  are dynamically natural as they occur as Julia sets of  quadratic polynomials for parameters from a generic set on the boundary
of the Mandelbrot set $\cal M$.
%  We prove that the  Hausdorff dimension of quadratic Julia sets depends continuously on $c$ 
%and is bi-Lipschitz at $-2$
%,  $\HD(J_c)-1\sim \sqrt{2+c}$ 
%for every Benedicks-Carleson parameter $c\in {\cal M}\cap \R$. 
\end{abstract}
\section{Introduction}
\subsection{Overwiew}
An intuition about a compact connected set is that if it oscillates  at every scale then its
Hausdorff dimension is strictly bigger than $1$.   One way to  quantify the concept
of geometric  oscillations is to use the theory of $\beta$-numbers of Bishop and Jones, see Definition~\ref{defi:1}.  A set $K$  is called  wiggly at $x$ and  scale $r>0$ 
if one can draw a triangle contained in the ball $B(x,r)$  with vertices in $K$ 
so that after rescaling to the unit ball, the triangle belongs to a compact family of triangles in $\R^d$, with $d\geq 2$. 
An immediate consequence of this approach is that  essentially, by the Pythagorean Theorem, 
the set $K$ ``accumulates'' an additional length at scale $r$.  
The word ``essentially'' is needed here because a wiggly continuum  can intersect $B(x,r)$ along  disjoint intervals. Even in the case of the regular intersections, the accumulation of  length at  scale $r$ remains true but the mechanism of a local length growth is more complicated and comes from the global hypothesis  about the connectivity of $K$.

In~\cite{bijo2}, it is proven that  every connected and compact planar set $K$ which
oscillates  uniformly at every scale around every point of  $K$ has the Hausdorff dimension
strictly bigger than $1$. The dimension  estimates of \cite{bijo2}  are   quantified in terms of 
$\beta$-numbers.
\begin{defi}\label{defi:1}
Let $K$ be a bounded set in $\R^d$ with $d\geq 2$, $x\in K$ and $r>0$.   
We define $\beta_{K}(x,r)$ by
$$ \beta_{K}(x,r):=  \inf_{L} \sup_{z\in K\cap B(x,r)} \frac{\dist{z,L}}{r}\ ,$$
where the infimum is taken over all lines $L$ in $\R^d$.
\end{defi} 
A  connected  bounded  set $K$ is called {\em uniformly wiggly} if 
$$\beta_{\infty}(K): = \inf_{x\in K}\inf_{r\leq \diam K} \beta(x,r)>0\;. $$
Theorem~1.1 of \cite{bijo2} states that if $K\subset \C$ is a continuum and $\beta_\infty(K)>0$
then $\HD(K)\geq 1+c\beta_{\infty}^{2}(K)$, where $c$ is a universal constant.
%The proof of Bishop and Jones' dimension result is  based on the observation that 
%length accumulation  in every scale preceding $r$
%leads to an exponential growth   of the length element at scale $r$.
%The key technical estimate is based on ``traveling salesman theorem'' of the second author \cite{jo1990}.
%We will build a theory  to study  the Hausdorff dimension of sets in euclidean spaces as a function 
%of an integral version of $\beta$-wiggliness. We will require only mild topological assumptions as compactness or connectivity.

One of the  main objectives  of this paper is to introduce analytic  tools based on  corona type constructions~\cite{len, gar, joncor} which can be useful in  the area of complex dynamics. In the area of dynamical systems one cannot expect that  generic systems are uniformly hyperbolic and that geometry
of invariant fractals can be estimated at every scale as it is required in ~\cite{bijo2}. However,  there are various results  showing that non-uniformly hyperbolic systems are typical in ambient  parameter spaces~\cite{jak, Becaa, rees, aspenberg, harm, smir}. Our results which rely only on 
mean  estimates of wiggliness  fit into a general scheme of studying metric properties of attractors for  generic dynamical systems.

A direct outcome of the proposed methods is Theorem~\ref{thmMain} which
gives  an integral and non-uniform version  of the dimension results of~\cite{bijo2}. 
Theorem~\ref{thminv} shows that the estimates of Theorem~\ref{thmMain} are 
sharp. The main technical difficulty in proving Theorem~\ref{thmMain} lies in controlling  non-wiggly portions of a continuum $K$ which can  have a finite length. If non-wiggly
portions of  $K$ have  infinite length then the assertion that  $\HD(K)>1$ is generally not true. 
As an example take  a unit
segment in the plane accumulated by a smooth curve winding infinitely many times around it.  
Admitting  exception sets allows also for  an immediate extension
of the dimension estimates of Theorem~\ref{thmMain}  over  those disconnected sets $K$ which can be
turned  into  a continuum $K\cup E'$ by  adding  a set $E'$ of finite length.

%consists  in a possibility   
%to extend  the dimension estimates of Theorem~\ref{thmMain} over some disconnected sets.
%Indeed, Theorem~\ref{thmMain} holds for  all sets  $K'$ such that  there exists a rectifiable curve $\gamma$  with the property that $K\cup \gamma$ is a continuum. 

In a different context,  a conceptually  similar approach was adopted by Koskela and Rohde in their proof that mean porosity replaces porosity for upper estimates of the Hausdorff dimension~\cite{koro}.

A standard observation about Hausdorff dimension is that only
asymptotic properties should intervene in the estimates. Another point  is  that even if  a connected set is not   wiggly at every scale, the property  to wiggle often enough, for example on a set of positive density of scales,  should be sufficient to observe an exponential growth
of  length in every scale.  

Theorem~\ref{thmMain}  generalizes Theorem~1.1  of~\cite{bijo2} in three directions. Firstly, it allows for  an exceptional set $E\subset K$ of finite $1$-dimensional Hausdorff measure. Secondly,
it is enough to  assume that  the continuum  $K$ oscillates at every point $x\in K\setminus E$ and every  scale $r>0$ with some  parameter $\beta_K(x,r)>0 $ which depends both on $x$ and $r>0$
so that 
\begin{equation}\label{mean}
\liminf_{r\rightarrow 0}
\frac {\int_r^{\diam\! K} \beta^2_K(x,t) \frac{dt}{t}}{-\log r}\geq \beta_0^2>0\;.
\end{equation}
This means that the uniform hypothesis  of ~\cite{bijo2} that there exists $\beta_0>0$ so that for every $x\in K$ and every $r<\diam K$,
$$\beta_K(x,r)\geq \beta_0$$
can  replaced by the integral condition $(\ref{mean})$ without affecting 
the main estimate on  $\HD(K)$ of ~\cite{bijo}  that $\HD(K)\geq 1+c\beta_0^2$, where $c$ is a universal constant.  Theorem \ref{thmMain} is also valid in higher ambient dimension.

An additional feature of Theorem~\ref{thmMain} is that  the condition (\ref{mean}) can be further relaxed, 
\begin{equation}\label{mean1}  
\liminf_{r\rightarrow 0}\frac {\int_r^{\diam\! K} \beta^2_K(x,r) \frac{dt}{t}}{-\log r}> 0\;
\end{equation}
and still obtain  a non-uniform estimate that  $\HD(K)>1$.

The main technical ingredient of the paper is Theorem~\ref{thmMainTech}  which claims that for every continuum $K$ of diameter $1$
there exists  a Borel probabilistic measure $\mu$ supported on a wiggly subset $Z$ of $K$  such that for every ball $B(x,r)$, $x\in Z$, 
$$\mu(B(x,r)) ~\leq ~ \frac{c'r}{\diam \!Z} \exp \left( -c\int_r^{\diam K} \beta^2_K(x,t)\frac{dt}{t}\right)\; ,$$
where $c,c'$ are universal constants.
The construction of the measure $\mu$ is based on a  combinatorial
Proposition~\ref{key} and a geometric version of the corona type construction  explained in~\cite{pajo}. 

The study of the distribution of the measure $\mu$ on  Julia sets is of independent interest as the relations between $\mu$ and other natural measures in complex dynamics are not known.
By Corollary~\ref{coro:con},   $\mu$ is absolutely continuous with respect to 1-dimensional Hausdorff measure $\H$ on connected Julia sets and thus $\mu$ is not atomic and $\HD(\mu)\geq 1$. 

\begin{theo}
\label{thmMain}
Suppose a nontrivial compact connected $K\ss\R^d$ where $d\geq 2$
is the union of two subsets  $K=W\cup E$,  $\H(E)<\infty$ and $\H(W)>0$.
\begin{itemize}
\item  If there exists $\be_0>0$ such that for all $x$ in $W$ 
$$\liminf_{r\ra 0}\frac{\int_r^{\diam \!K}\be_K^2(x,t)\frac{dt}t}{-\log r}\geq \be_0^2,$$
then $$\HD(K) \geq 1 + c \beta_0^2, $$
where $c$ is a universal constant.
\item If for all $x\in W$
$$\liminf_{r\ra 0}\frac{\int_r^{\diam \!K}\be_K^2(x,t)\frac{dt}t}{-\log r} > 0,$$
then $$\HD(K) > 1.$$
\end{itemize}
\end{theo}

%\paragraph{About the hypotheses of Theorem~\ref{thmMain}.}
The condition $\H(E)<\infty$ can not be further relaxed as shows
the following example (Warsaw sine): Let $K$ be  the closure of the graph $G$ of $y=\sin (1/x)$,
$x\in (0,1]$, in the Euclidean planar topology. The wiggly set $W$ is a vertical segment $\{0\}\times
[-1,1]$, the exceptional set $E$ is the graph $G$ which has an infinite length. Clearly, $\HD(K)=1$.
%For any bounded set $N$ in the plane and $x\in N$, $\beta_N(x,r)$ close to $0$ means that
%the set $N$ is $\varepsilon$-porous at $x$ with $\varepsilon\sim 1/2$. 
Further examples showing that all hypotheses of Theorem~\ref{thmMain} are essential
are discussed in Section~\ref{sec:example}.

We have already observed that the hypothesis about connectivity of $K$ can be slightly relaxed.
%by a weaker condition that $K$ can be made connected by adding a  set of the finite $\H$ measure .
In fact,  the estimates of Theorem~\ref{thmMain} can be also localized. If $D(x,R)$ is a ball such that $K':= \D(x,R)\cap K \cup  \partial D(x,R)$ is connected we can apply Theorem~\ref{thmMain} for a new  continuum $K'$ and a new exceptional  set  $E'= \D(x,R)\cap E\cup \partial \D(x,R)$. Clearly,
$\HD(K')=\HD(K\cap \D(x,R))$.
%According to  Corollary~\ref{coro:weak}, the hypothesis about  connectivity of $K$ in Theorem~\ref{thmMain} can be further  relaxed.
%by the  condition   that there is $W'\subset W$ with the property that $\H(W')>0$ and for  every $z\in W'$ there is a ball $\D(z,R)$ such that every connected component of $\D(z,R)\cap K$ has the diameter bigger
%than a uniform  constant multiple of $R$.  
% The new hypothesis allows for infinitely many connected components of $K$.  A natural example is a compact  set with a non-trivial continuum of convergence.  
% $K\times [0,1]$,  where $K\subset[0,1]$ is a standard $\frac{1}{3}$-Cantor set. 
%but, for example, it is enough to assume that all connected components of $K$ have the diameter  bigger than some positive constant. 
% to obtain the claim of Theorem~\ref{thmMain}.

On the other hand,  there are  obvious examples of  totally disconnected and uniformly  wiggly compacts  $K\times K$, where $K\subset [0,1]$
is a Cantor set of bounded geometry with  $\HD(K)< \frac{1}{2}$.   
\paragraph{Almost flat sets.}
We will show that the estimates of Theorem~\ref{thmMain} are sharp. 
Recall that a set $N\subset\C$ is $\varepsilon$-porous at $x$ at scale $r>0$  if 
there is $z\in \D(x,r)$ such that  $\D(z,\varepsilon r)\subset \D(x,r)\setminus N$. 
The condition  $\beta_N(x,r)\leq \alpha$ implies that  $N$ is  $(1-\alpha)/2$-porous at scale $r$ at $x$.

If one assumes that for every $x$ in a bounded set $N$,
\begin{equation}\label{equ:inf}
\limsup_{r\ra 0}\frac{\int_r^{\diam \!N}\be_N^2(x,t)\frac{dt}t}{-\log r}=0,
\end{equation}
then  $\HD(N)\leq 1$ follows  from the dimension results  of Belaev and Smirnov for mean porous sets, Corollary~1 in ~\cite{besm}.
%which states that asymptotic porosity  $1/2$ on the sets of integers of density $1$ at every point of a bounded set $N$ 
%yields  
The condition  $\beta_N(x,r)=0$ is much stronger than the maximal porosity $1/2$ at $x$. It turns out that one  can replace $\limsup$ by $\liminf$
in the inequality~(\ref{equ:inf}) and still obtain that $\HD(N)\leq 1$. 
\begin{theo}
\label{thminv}
Let $N$ be a set in $\R^d$ with $d\geq 2$ such that for all $x\in N$,
$$\liminf_{r\ra 0}\frac{\int_r^{1}\be_N^2(x,t)\frac{dt}t}{-\log r}\leq \beta_0^2 ,$$
then there is a universal constant $c>0$ such that $$\HD(N) \leq 1+c\beta_0^2 .$$
\end{theo}
\begin{proof}
Without loss of generality we may assume that $\diam N=1$.
Fix $\epsilon >0$ and $n_0>1$. For every $n\geq n_0$, denote
$$ X_{n}=\{x \in N: \int_{2^{-n}}^{1}\be_N^2(x,t)\frac{dt}{t}\leq (\epsilon +\beta_0^2) n \log 2\}\;.$$
By definition, $N=\bigcup_{n\geq n_0} X_{n}$.  %Associate to every point in $X_n$ a ball $B_x=\D(x,2^{-n})$.
%One can check that for any $x$ from the closure $\overline{X_n}\subset N$
%we have that 
%$$\int_{2^{-n}}^{\diam \!X_n}\be_{X_n}^2(x,t)\frac{dt}{t} \leq 
% 4\int_{2^{-n}}^{\diam \!N}\be_N^2(x,t)\frac{dt}{t}
%\leq -4 \epsilon n\;.$$
By Besicovitch's covering theorem, we can find  $M$  subcollections ${\cal G}_i$ such that every two balls from the same subcollection
${\cal G}_i$ are disjoint, every ball is of radius  $2^{-n}$, and  $X_n$ is covered by  the balls from ${\mathcal G}(n)=\bigcup_{i\leq M}{\mathcal G}_i$. 
Let $Z_i$ be the set of the centers of the balls from  ${\mathcal G}_i$. It is a finite set with the property that
$$\int_{0}^{\diam \!Z_i}\be_{Z_i}^2(x,t)\frac{dt}{t} \leq 
\int_{2^{-n}}^{1}\be_N^2(x,t)\frac{dt}{t}
\leq  (\epsilon+\beta_0^2) n \log 2\;.$$
By Theorem~\ref{thmBJ97}, for every $i\leq M$, the set $Z_i$ is contained in a curve $\Gamma_i$ 
of length $\H(\Gamma_i)\leq Ce^{C(\epsilon +\be_0^2)n \log 2}$, where $C>1$ is a universal constant.  Using the fact that the balls from ${\mathcal G}_i$ are disjoint, we have that
$$\sum_{B\in {\cal G}_i} (\diam B)^{1+\alpha} \leq 2^{(-n+1)\alpha}\; \H(\Gamma_i)\leq C_1
2^{-n(\alpha -C(\epsilon+\beta_0^2))}\;.$$
Hence, the sum of the diameters to the power $1+\alpha$ of the balls from ${\mathcal G}(n)$ is smaller than
$$ M C_12^{-n(\alpha -C(\epsilon+\beta_0^2))}\;.$$ 

Let ${\cal G}=\bigcup_{n\geq n_0}{\mathcal G}(n)$. Assuming that $\alpha$ is bigger than $2 C(\epsilon+\beta_0^2)$, we obtain that
$$\sum_{B\in {\mathcal G}} (\diam B)^{1+\alpha}\leq \sum_{n=n_0}^{\infty} M C_1 2^{-n(\alpha -C(\epsilon+\beta_0^2))} \leq C' 2^{-n_0 C\beta_0^2}\;,$$
where $C'$ is a  constant. Passing with $n_0$ to $+\infty$, we infer  that ${\mathcal H}^{1+\alpha}(N)<+\infty$ which completes the proof.

\end{proof}

%\paragraph{Example.} The condition $\H(E)<\infty$ can not be relaxed further as shows
%the following example (Warsaw sine): Let $K$ be  the closure of the graph $G$ of $y=\sin (1/x)$,
%$x\in (0,1]$, in the Euclidean planar topology. The wiggly set $W$ is a vertical segment $\{0\}\times
%[-1,1]$, the exceptional set $E$ is the graph $G$ which has an infinite length. Clearly, $\HD(K)=1$.
\paragraph{Invariance property.} Suppose that $K$ is a continuum which satisfies the hypothesis of Theorem~\ref{thmMain} with $E=\emptyset$. Let $\mathcal{S}_K$ be the set of all continuous functions  $h:U\mapsto \C$, defined on some 
neighborhood $U$ of $K$, conformal with the Jacobian  different from
zero on $K\setminus F\not = \emptyset$ and $\H(h(F))=0$. We recall that $f:U\mapsto \C$ is conformal at $z_0\in U $
if the limit $(f(z)-f(z_0))/(z-z_0)$ exists and is different from $0$. Then, 
$$\inf_{h\in \mathcal{S}_K}\HD(h(K))\geq 1+ c\beta_0^2\;.$$
Indeed, $h$ does not decrease $\beta_K(z)$ at any point $z\in K\setminus F$. Since $h(K)$ is a continuum, the assumption that $\H(h(F))=0$ implies that $h(K\setminus F)$
is a  wiggly part of $h(K)$, see Theorem~\ref{thmMean},  of positive length.

\paragraph{Generalizations are possible.}
A direction for possible generalizations
can be adopted following~\cite{guy2} where a uniform non-flatness of compacts with respect to $d$-dimensional planes is studied. Under some additional topological assumptions which replace connectedness, it is proved that the Hausdorff dimension of uniformly non-flat compacts in $\R^n$ is strictly bigger than $d$.

\subsection{Mean wiggly continua}
The integral condition $(\ref{mean})$ has a discrete counterpart. Let $K$
be a  set in $\R^d$ of diameter $1$ and fix $\lambda\in (0,1)$.  For every $x\in K$ and $m\in \N$,
if $\beta(x,\lambda^{-m})>c_0$ then $\beta(x,\lambda^{m-1})>\lambda c_0 $. Hence,  if $r\in A_m=[\lambda^{-m-1}, \lambda^{-m})$,
$m\in \N$, then the limit 
\begin{equation}\label{mean3}
\liminf_{r\ra 0}\frac{1}{-\log r} \int_r^{1} \be_K^2(x,t)\frac{dt}{t}
\end{equation}
is equivalent to 
\begin{equation}\label{mean2}
\liminf_{m\ra \infty} \frac{1}{m} \sum_{i=0}^{m-1} \sup_{t\in A_i}  \beta^2(x,t)
\sim \liminf_{m\ra \infty} \frac{1}{m} \sum_{i=0}^{m-1}   \beta^2(x, \lambda^{-i}) \;.
\end{equation}

We will say that a  set $K\subset \R^d$ is {\em mean wiggly} at a point $x\in K$ with parameters $\lambda, \kappa, \beta_0 \in (0,1)$ if there is  $r(x) > 0$ such that for $\lambda^{n} < r(x)$ the number of wiggly  scales 
$\lambda^{m}, m < n$, is greater than $\kappa n$. Here, a wiggly  scale is defined by the condition that $\beta(x, \lambda^{m}) \geq \beta_0 $

In particular, the limit~(\ref{mean2})  is positive iff  $K$ is mean wiggly with some positive
parameters at $x$.
%$r(x) > 0$ and $\beta_0,\kappa\in (0,1)$  such that  for $2^{-n} < r(x)$, the number of  scales  $2^{-m}, m < n$,  with the property that $\beta(x, 2^{-m}) > \beta_0$,  is greater than $\kappa n$.  
%
%We propose the following definition of {\em mean wiggliness}.
%\begin{defi}\label{mwc}
%We say that a  set $K\subset \C$ is mean wiggly at a point $x\in K$ with parameters $\lambda, \kappa, \beta_0 \in (0,1)$ if there is  $r(x) > 0$ such that for $\lambda^{n} < r(x)$ the number of wiggly  scales 
%$\lambda^{m}, m < n$, is greater than $\kappa n$. Here, a wiggly  scale is defined by the condition that
%$\beta(x, \lambda^{m}) \geq \beta_0 $.
%\end{defi}
The hypothesis of Theorem~\ref{thmMain} can be reformulated in terms of mean density of wiggly scales.
\begin{theo}\label{thmMean}
Suppose that $K\ss\R^d$ with $d\geq 2$ is a continuum of diameter $1$ and $K=W\cup E$ where  $W$ and $E$ 
satisfy the following:
\begin{itemize}
\item  $\H(E)<\infty$ and $\H(W)>0$.
\item  $K$ is mean wiggly at every point $x\in W$.
\end{itemize}
Then $$\HD(K)>1\;.$$

If additionally, the parameters of the wiggliness $\lambda, \kappa, \beta_0$  at $x\in W$ are uniform, that is do not depend on $x\in W$, then
$$\HD(K) \geq 1 + c'  \lambda^4 \beta_0^2 ~\kappa, $$
where $c'$ is a universal constant.
\end{theo}
\begin{proof}
By Theorem \ref{thmMain}, it is enough to estimate from below the limit~$(\ref{mean3})$
for all $x\in W$.  
For every integer $i\geq 0$, define $A_i =[\lambda^{i+1}, \lambda^{i})$. Set $\chi_i=1$
if $K$ is  $\beta_0$-wiggly at $x\in W$
at some scale from $ A_{i}$. If $\chi_{i+1}=1$ then
$$
\int_{A_{i}}\be_K^2(x,r)\frac{dt}{t}\;
\geq \; \int_{\lambda^{i+1}}^{\lambda^{i}} 
\left(\beta_0\frac{\lambda^{i+2}}{t}\right)^2 \frac{dt}{t}\;  \geq \; 
\lambda^4 \beta_0^2 \log \lambda^{-1} \;.$$

Therefore, the lower bound of the limit $(\ref{mean3})$ is given by 
$$\lambda^4 \beta_0^2 \log \lambda^{-1}\; \liminf_{n\rightarrow \infty}  \frac{\# \{\chi_i=1: i\in [0,n+1)\} }{(n+1)\log \lambda^{-1} }
\geq  \kappa \lambda^4 \beta_0^2  \;.$$
\end{proof}
%To formulate a discrete version of Theorem~\ref{thminv} we will need a concept of {\em mean flatness}. 
%\begin{defi}\label{flat}
\paragraph{Almost flat sets.}
A  set $K\subset \R^d$ is {\em almost flat} at a point $x\in K$ with parameters $\lambda, \kappa, \beta_0 \in (0,1)$ if there is a sequence of positive integers $(N_i)$   such that for
every $i\in \N$  the number of flat  scales 
$\lambda^{m}, m < N_i$, is greater than $\kappa N_i$.  A flat scale is defined by the condition 
$\beta(x, \lambda^{m}) \leq \beta_0 $.
%\end{defi}
\begin{theo}\label{thminvmean}
Suppose that $K\subset \R^d$, $d\geq 2$, is a set of diameter $1$ and $K$ is almost  flat at every point $x\in K$ with the parameters $\lambda, \kappa, \beta_0$  that  do not depend on $x\in W$.
Then, 
$$\HD(K) \leq 1  + c'\lambda^{-4}(1-\kappa+  \beta_0^2 ~\kappa), $$
where $c'$ is a universal constant.
\end{theo}
\begin{proof}
The proof follows immediately from Theorem~\ref{thminv} and a short calculation very much 
the same as in the proof of Theorem~\ref{thmMean}.
\end{proof}
%If $K\ss\C$ and
If for every $x\in K$, the sequence $(N_i)$ from the hypotheses of Theorem~\ref{thminvmean} 
contains the set of almost all  positive integers,  then 
Corollary~1 in~\cite{besm} (see also~\cite{salli,matilla}))  implies  that
$$\HD(K)<d-\kappa +\frac{C}{|\log \beta_0|}\;,$$
where  $C$ is a universal constant. 

\paragraph{Non-porous continua.}
Sharp estimates of the Hausdorff dimension usually require that a set has a self-similar or at least a well-defined hierarchical structure.  If this additional structure is present
then a  relative  geometrical data (scalings)  leads to useful dimension estimates,
see  for  example ~\cite{area} for applications  of this technique in the complex dynamics. When the hierarchical structure of the set is missing, the dimension estimates become difficult.  Our objective is to prove sharp lower bounds  for the Hausdorff dimension under mild topological  assumptions   as connectivity,  compactness, and relative density.  
In Section~\ref{sec:com},  we further discuss possible applications of our techniques in the case of compact sets.

%We will prove an asymptotically precise estimate of the  Hausdorff dimension for  continua  which are not porous 
%at scales of positive density.  %There are no additional assumptions about the organization  of the set. 

Let $\epsilon\in (0,1/2)$ and $\lambda \in (0,1)$.  The set $K$ is not $\epsilon$-porous at $z\in K$ at scales of density $\kappa>0$ if 
$$\liminf_{m\rightarrow\infty}\frac{1}{m} 
\mbox{$\#\{n\in (0,m): K$ is not $\epsilon$-porous at scale 
$\lambda^{n}$ at $z$}\}\geq \kappa\;.$$

\begin{theo}\label{theo:nonpor}
Let $K$ be a continuum in $R^{d}$, $d\geq 2$. Suppose that $K=W\cup E$ such that $\H(W)>0$ and $\H(E)<\infty$.
If there are  positive numbers $\lambda, \epsilon, \kappa \in (0,1)$ 
such that for every $z\in W$ the set $K$ is not $\epsilon$-porous at $z$ at scales of density $\kappa$,
%there is  $r(z) > 0$ such that for $\lambda^{n} < r(x)$ the number of non-porous   scales 
%$\lambda^{m}, m < n$, is greater than $\kappa n$. Here, a non-porous  scale is defined by the condition that 
%$K$ is not $\epsilon$ porous at any scale $\omega\in (\lambda^{n+1},\lambda^{n})$.
then
$$\HD(K)\geq 1+\kappa \left(d-1-\frac{C}{|\log \epsilon|}\right),$$
where $C>0$ is a constant depending only on $\la$ and $d$.
\end{theo}
In Section~\ref{sec:theo5}, we will provide examples showing that all hypotheses of Theorem~\ref{theo:nonpor} are essential.
%about connectivity and $\H(W)>0$ 

%We devide each square $Q_i^n$ by drawing horizontal lines into $n$ isometric rectangles. Denote the union of the boundaries
%of these rectangles by $R^i^n$.    

\subsection{Compact sets}\label{sec:com}
The mass  distribution principle is one of the basic techniques
for Hausdorff dimension. The method is however not direct as
one needs to construct   a probability measure $\nu$ supported on $K$  with suitable scaling properties to get  lower bounds of $\HD(K)$,
$$
\nu(B(x,r))\leq r^s \;\;\; \mbox{ for all $x\in K$ and all $r>0$} \Longrightarrow  
\HD(K)\geq s\; .
$$

For self-similar fractals, as one-third Cantor set, or more generally sets of ``bounded geometry'', 
the method leads to  precise estimates of Hausdorff dimension. In general, for more complicated  sets which show an ``unbounded geometry'',  the mass distribution principle
encounters difficulties  as the geometry  of a set is not controlled in every scale, as it is the case
for self-similar fractals,  but only in some and often scarcely distributed  scales. 

We propose a direct geometric method to produce universal lower bounds for Hausdorff dimension of compact sets. To this aim,  we will define a notion of {\em convex density}. 

\begin{defi}
Let $K$ be a subset of $\R^d$ with $d\geq 1$ and $x\in K$. We define a convex density
$d_K(x, r)$ of $K$ at  $x$  at the scale $r>0$ as
$$d(x,r)= \frac{I(x,r)}{2r}\;,$$
where $I(x,r)$ is the diameter of the convex hull of $K\cap  B(x,r)$.
\end{defi}

\begin{theo}\label{Cantor}
Let $K\subset \R^d$ with $d\geq 1$ be a compact set.  Suppose that for every $x \in K$  we have that
$$\liminf_{r\rightarrow 0} \frac{\int_r^{\diam{K}} d_K^2(x,t)\frac{dt}{t} }{-\log r}>0 \;.$$
Then $$\HD(K)>0\;.$$

If additionally, there is a constant $d_0>0$ such that for every $x\in K$,
\begin{equation}\label{d}
\liminf_{r\rightarrow 0} \frac{\int_r^{\diam{K}} d_K^2(x,t)\frac{dt}{t} }{-\log r}\geq d_0^2 \;,
\end{equation}
then 
$$ \HD(K)\geq c d_0^2\;,$$
where $c$ is a universal constant.
\end{theo}
\begin{proof}
Without loss of generality, $K\subset [0,1]^d$.
We build a continuum $K^*\subset \R^{d+1}$ by joining the point $S=(1,\ldots,1)\in \R^{d+1}$ with every $x\in K$ by the segments $I_x$ with the endpoints at  $S$ and  $x$, 
$$ K^*=\bigcup_{x\in K} I_x\;.$$
Since every horizontal hyperplane $x_{d+1}=w$, $w\in (0,1)$, intersects $K^{*}$ along an affine copy
of $K$, the condition 
$$\liminf_{r\ra 0}\frac{ \int_r^{\diam \!K} \be_K^2(x,r)\frac{dt}t}{-\log r}\geq \beta_0^2$$
is satisfied for every $z\in K^*\setminus S$ and some $\be_0= c' d_0$ where $c'$ is a universal constant.
$K^*$ is also wiggly at $S$ but we rather refer to the obvious fact that $\H(S)=0$,  Theorem~\ref{thmMain} implies that 
$$\HD(K^{*})\geq 1+c \beta_0^2=1+c c'^2 d_0^2\:. $$
By the product formula (Theorem~8.10 in~\cite{matilla}), 
\begin{equation}\label{equ:prod}
1+c c' d_0^2 \leq  \HD(K^{*}) \leq \HD(K) + \MD([0,1])= \HD(K) +1\;.
\end{equation}
\end{proof}
In general, Theorem~\ref{Cantor} admits only finite $E$ as exceptional sets.
Indeed, a countable compact $K=\{0\}\cup \bigcup_{n=1}^{\infty} \{\frac{1}{n}\}$ has positive convex density at all scales only at $0$ and $\HD(K)=0$.

The estimates of~Theorem~\ref{Cantor} are sharp as shows the following theorem.
\begin{theo}\label{Cantorbis}
Let $N\subset \R^d$ with $d\geq 1$.  Suppose that for every $x \in N$  we have that
$$\liminf_{r\rightarrow 0} \frac{\int_r^{1} d_N^2(x,t)\frac{dt}{t} }{-\log r}\leq d_0^2 \;.$$
Then 
$$ \HD(N)\leq  c d_0^2\;.$$
where $c$ is a universal constant.
\end{theo}
\begin{proof}
Without loss of generality we may assume that $N\ss[0,1]^d$.
Consider $K=N\times [0,1]\subset \R^{d+1}$. The set $K$ satisfies the hypothesis of Theorem~\ref{thminv} with $\beta_0=d_0$.
Therefore, there exists a universal constant $c>0$ such that
$$\HD(K)\leq 1+c\beta_0^2\;.$$
By the product formula for Hausdorff dimension (Theorem~8.10 in~\cite{matilla}),
$$\HD(N\times [0,1])\geq \HD(N)+\HD([0,1])=\HD(N)+1\;$$
and thus, $\HD(N)\leq cd_0^2$.  
\end{proof}

\paragraph{Invariance property.} Suppose that $K$ is a real compact which satisfies the hypothesis of Theorem~\ref{Cantor}.
Let $\mathcal{D}_K$ be the set of all continuous real functions $h$ defined on some 
neighborhood  of $K$, with the property that there exist $\theta_1,\theta_2>1$ such that for every $x\in K$ and every interval $I$
with the middle point at $x$ and the lenght $|I|$ small enough, 
\begin{equation}\label{equ:qc}
\theta_1 |h(I)|\leq |h(2I)|\leq \theta_2 |h(I)|,
\end{equation}
where $2I$ stands for the double of $I$ and $|h(I)|$ denotes the length of the image $h(I)$.
Then, 
$$\inf_{h\in \mathcal{D}_K}\HD(h(K))\geq 1+ c(\theta_1,\theta_2) d_0^2\;,$$
where $c(\theta_1,\theta_2)$ a constant which depends only on $\theta_1$ and $\theta_2$. Morever, if $\theta_1$ and $\theta_2$
tend to $2$ then $c(\theta_1,\theta_2)$ tends to a universal constant. The constant $c(\theta_1,\theta_2)$ can be easily
estimated using  Theorem~\ref{thmMean} and the fact that the condition $(\ref{equ:qc})$ prohibits, on one hand, too much expansion of non-wiggly scales and, on the other  hand, too much contraction of wiggly scales (see also the proof of Theorem~\ref{TCE}).

Every differentiable homeomorphism $h:U\mapsto \R$ with the derivative different from $0$ belongs to $\mathcal{D}_K$. Another example
of the class of maps in  $\mathcal{D}_K$ are real quasi-regular functions.

\paragraph{Example.} The $\frac{1}{3}$-Cantor set is obtained from the unit interval $X_0=[0,1]$ through an inductive procedure.
The closed set $X_n$ is obtained from the closed set $X_{n-1}$ by removing the middle one-third of each component of $X_{n-1}$. 
The $\frac{1}{3}$-Cantor set is $X=\bigcap_{n=0}^{\infty} X_n$. The Hausdorff dimension of $X$ is $\HD(X)=\frac{\log 2}{\log 3}$.
One can easily check that $X$ is of positive convex density, that is for every $x\in X$, we have that  $d_X(x,r)\geq d_0=\frac{1}{4}$ provided  $r$ is small enough. Therefore, there exists a universal constant
$c$ such that the image of $X$ by any real diffeomorphism is bigger than $c$.  

\subsection{Bowen's dichotomy}\label{sec:bowen}
The property of uniform wiggliness was used in~\cite{bijo2} to prove Bowen's dichotomy for
a connected limit set of an analytically finite and not elementary  Kleinian group. 
In  polynomial  dynamics,  the corresponding result was proven earlier by  Zdunik, ~\cite{z}.
The connected Julia set of a polynomial is either a segment/circle or its Hausdorff dimension  is strictly bigger than $1$. The
strategy of the proof in~\cite{z}  was  different than  that of Bishop and Jones
and was based on the study of the statistical properties  of 
the unique   measure $\mu$ of maximal entropy for $f$
( in polynomial dynamics $\mu$ coincides with the harmonic measure with the pole at $\infty$).
The main observation of \cite{z} is that $\phi\circ f^{n}$, where $\phi: =\HD(\mu) \log |f'|-\log (\deg f)$, can be treated as  a sequence of random variables. There are two possibilities,  either $\phi$ is homologous to $0$ in $L^{2}(J,\mu)$ or the law of iterated logarithm holds. The former  case,  by  the``boot strapping'' argument,  leads to an analytic Julia set, while the latter implies  that $\HD(\mathcal{J})>1$. 
% The theorem of Zdunik
%does not provide  explicit lower bounds on  the Hausdorff dimension of $\mathcal{J}$, see  Section~\ref{sec:mand} for a further discussion.

\paragraph{Topological Collet-Eckmann rational maps.} The definition of rational TCE maps (see \cite{przro}) is usually  stated as follows.
Let $f$ be a rational map of the Riemann sphere $\CC$ of degree bigger than $1$.
For a given positive  $\delta$ and $ L$ define
${\cal G}(z, \delta, L)$ to be the set of positive
integers $n$ such that 
\[\#\{i: 0\leq i\leq n, \Crit_f \cap  \mbox{Comp}_{f^{i}(z)}
f^{-n+i}(B(f^{n}(z),\delta))\not = \emptyset\}\leq L,\]
where Comp$_{y}(A)$ stands for the component of $A$ which contains $y$.

\begin{defi}\label{defi:tce}
A rational map  $f$ satisfies TCE 
if there are positive $\delta$, $L<\infty$ and $\kappa$
so that for every point $z\in \CC$ ($z$ belongs to the Julia set) we have that
\[\inf _{n} \frac{\# {\cal G}(z,\delta, L)\cap [1,n]}{n}\geq \kappa\;.\]
\end{defi}

We have the following geometric counterpart of Bowen's dichotomy.

\begin{theo}\label{TCE}
Suppose that $f$ is a rational TCE map of degree bigger than $1$.  If the Julia set $\cal {J}\not  =\CC$ of
$f$ is connected then ${\cal J}$  is either an interval/circle or a mean wiggly  continuum.
\end{theo}

%\paragraph{Remark.} TCE maps can have only attracting invariant Fatou components. Theorem~\ref{TCE} can be strenghtened to claim that if ${\cal J}\not = \CC$ then  the boundary of every invariant Fatou component is mean wiggly unless ${\mathcal J}$ is an interval/circle. The proof
%of this stronger version follows closely the proof of Theorem~\ref{TCE}.\smallskip

%{\em {\small {\underline{Question}: Is it true that every connected component of the Julia set
%of a TCE rational map is either mean wiggly or an analytic set?}}}\\

\subsection{Harmonic measure and mean wiggly Julia sets}\label{sec:mand}
Let $E$ be a full compact in $\C$.
The harmonic measure $\omega$ of $E$ with a base point at $\infty$
can be described in terms of the Riemann
map 
\[ \Psi :\: \C\setminus \ol{D(0,1)} \mapsto \C\setminus E\]
which is tangent to identity at $\infty$. 
Namely,  $\Psi$ extends radially almost everywhere
on the unit circle with respect to the normalized $1$-dimensional
Lebesgue measure $d\theta$ and  $\omega=\Psi_{*}(d\theta)$.
%The harmonic measure has a probabilistic interpretation, namely
%if $E$ is a subset of $E$ then $\omega(E)$
%is a hitting probability of $E$ by a Brownian particle sent from
%$\infty$. 

The Mandelbrot set ${\cal M}$ is the set of the parameters $c\in \C$ for which the corresponding Julia
set ${\cal J}_c$ of the quadratic polynomial $f_c(z)=z^2+c$ is connected. It is known that both the Mandelbrot
set and its complement are  connected. A parameter $c\in \partial  {\cal M}$ is called {\em Collet-Eckmann}
if    $\liminf_{n\rightarrow \infty} \frac{\log|(f_c^{n})'(c)|}{n}>0\;$.  

It was proven in ~\cite{harm, smir} that 
the Collet-Eckmann parameters in the boundary of the Mandelbrot set are of full harmonic measure.
Since every Collet-Eckmann quadratic polynomial  $f_c(z)=z^2+c$ satisfies TCE property, we can invoke
Theorem~\ref{TCE} to derive the following corollary. 
\begin{coro}\label{co}
For almost all $c\in \partial {\cal M}$, the corresponding Julia set is a mean wiggly continuum.
\end{coro}
\paragraph{Remark.} The claim of Corollary~\ref{co} remains true for unicritical polynomials $z^{d}+c$, $d\geq 2$, with $c$ from a generic set on the boundary of the connectedness locus $\mathcal {M}_d$, see ~\cite{harm, smir}.
\subsection{Various examples}\label{sec:example}

\textbf{Example 1.} We describe an example where $E$ has finite positive $1$-measure (denoted by $\H(E)$) with $\HD(W)=1$ and $\H(W)=0$.

This is a modified version of the four corners Cantor set. Let $S_0=[0,1]^2$ be the unit square. Let $S_1^1,\ldots,S_1^4$ the disjoint sub-squares of side-length $0<\frac14a_1<\frac 12$ which have each a common corner with $S_0$. We repeat the construction inside each square of $n$-th generation to obtain four squares of side length $|S_{n+1}|=\frac14a_{n+1}|S_n|$. We set
$$W=\bigcap_{n\geq 1}\bigcup_{i=1}^{4^n}S_n^i.$$

Observe that $|S_n|=4^{-n}\prod_{i=1}^na_n$. We let $a_n\nearrow 1$ and show that $\HD(W)=1$. We may define a measure $\mu$ supported on $W$ such that $\mu(S_n^i)=4^{-n}$. Then
$$\lim_{n\ra+\infty}\frac{\log\mu(S_n)}{\log|S_n|}=\liminf_{n\ra+\infty}\frac{n\log4}{n\log4+\log a_1+\ldots+\log a_n}=1.$$
As $\frac{|S_n|}{|S_{n+1}|}$ is bounded, Billingsley's lemma shows that $\HD(W)=1$. The same bound gives a lower bound of $\be_W(x,r)$ for all $x\in W$ and $r>0$.

The set $W$ can be covered by the $4^n$ squares of $n$-th generation. Therefore it is enough to have $\lim_{n\ra +\infty}\prod_{i=1}^na_i= 0$ to obtain that $\H(W)=0$.

Let $E$ be the union of all diagonals of all squares. 
We may easily compute that 
$$\H(E)=2\sqrt 2(1+\sum_{n\geq 0}4^n|S_n|).$$
We set $a_n=\frac{n^2}{(n+1)^2}$ so $\H(E)<+\infty$. Note that this sequence satisfies the previous conditions.

The set $K=W\cup E$ is a continuum with the desired properties. $K$ is also locally connected.\\

\textbf{Exemple 2.}
If in the previous example we set $a_n=1$  for every $n\geq0$ (the standard four corners Cantor set) then $\H(W)$ is finite,  $\H(E)=\infty$, and $\HD(W\cup E)=1$. The continuum  $K=W\cup E$ is locally connected.\\

\textbf{Example 3.} We give an example  of a locally connected continuum such that $\HD(K)=1$ and  $\H(W)>0$ (but $\H(E)=+\infty$).

Let $W=[0,1]$ and $E$ the union of vertical segments of length $2^{-n}$, with center $k2^{-n}\in W$, for all $n>0$ and $0<k<2^n$, $k$ odd. The continuum $K=W\cup E$ is locally connected and $\HD(K)=1$ as a countable union of segments. It is not hard to see that $\be_K(x,r)$ has a uniform lower bound for all $x\in W$ and all $0<r<1$ (but $\be_W(x,r)=0$). 

\section{Density of wiggliness}
We want to construct a probability measure $\mu$ which captures wiggliness  of a continuum $K$ and allows for effective lower bounds of 
$\HD(K)$. Let $F$ be the set of points of $K$ which are not wiggly
$$F:=\{x\in K : \int_0^1\be_K(x,t)^2\frac{dt}t<+\infty\}.$$
We want to construct $\mu$ which vanishes on  $F$. In this case, $\mu$ will be  generally scale dependent as wiggly parts of $K$ can be contained in a ball of arbitrary small radius. 

The main ingredient of the proof of Theorem~\ref{thmMain} is  the following result.

\begin{theo}
\label{thmMainTech}
Let $K\ss\R^d$ with $d\geq 2$ be a connected compact of diameter $1$ and $F\se E\se K$ with $\H(E)<+\infty$ and $\H(K\sm E)>0$.
There exist a universal constant $C'>0$, a constant $C>0$ depending only on $d$ and a Radon probability measure $\mu$ supported on $B(y,R)\cap K$, $y\in K$ such that $\mu(E)=0$ and for all $x\in K$ and $r>0$
$$\mu(B(x,r))\leq R^{-1}Cr\exp\left(-C'\int_r^1\be_K^2(x,t)\frac{dt}t\right).$$
\end{theo}

The proof of Theorem~\ref{thmMainTech} is technically involving. We start by proving an important  combinatorial result, Proposition~\ref{key}, 
and then follow the construction from the proof of Theorem 45 in \cite{pajo}. This construction was proposed by the second author and after some modifications made available
in the written form by David in~\cite{guy} around a decade ago. Proposition~\ref{key}
replaces a {\em stopping time} argument,  usually needed in corona type constructions,
by a direct estimate of length of a crossing curve which is wiggly at many scales.
% It also allows to complete the proof of the key estimate (25) on page 47  in~\cite{pajo}. 

%For the sake of simplicity, the  proof of Theorem~\ref{thmMainTech} is  given in details only  for planar sets $K$. We indicate  what modifications are needed for the general case. 

\paragraph{Planar continua.}
For the sake of simplicity, the  proof of Theorem~\ref{thmMainTech} is  given in details only  for planar sets $K$. We indicate  what modifications are needed for the general case. 

As a preparation, we need to state some results and prove a few facts about the length (or $1$-Hausdorff measure) of wiggly continua in the plane. 

Let $Q\ss \C$ be a square (with sides parallel to the axis). Unless specified otherwise, squares are considered to contain only the left and top edges. Let $|Q|$ denote the side length of $Q$. Let $\del(Q)$ denote the set of all dyadic sub-squares of $Q$ and $\del_k(Q)$ those with side length $2^{-k}|Q|$. For any $\lam > 0$, let $\lam Q$ be the square with the same center as $Q$ and with $|\lam Q|=\lam |Q|$. For a set $S$ of squares, $|S|:=\sum_{Q\in S}|Q|$ and $\#S$ is its cardinal.

For any $x\in Q_0$, where $Q_0:=[0,1]\times[-\frac12,\frac12]$, let 
$$\del(x)=\{Q\in\del(Q_0)\;|\;x\in Q\}.$$ 
Let us also write $\del$ for $\del(Q_0)$ and $\del_k$ for $\del_k(Q_0)$.

\begin{defi}\label{defi:2}
Let $K$ be a compact set in the plane and $Q$ a square such that $Q\cap K\neq\es$.
We define $\beta_{K}(x,r)$ by
$$ \beta_{K}(Q):=  \inf_{L} \sup_{z\in K\cap 3Q} \frac{\dist{z,L}}{|Q|}\ ,$$
where the infimum is taken over all lines $L$ in the plane.
\end{defi}

For a compact set $K\se Q_0$ is easy to check that 
\begin{equation}
\label{equIS}
\beta_\infty^K(x):=\int_0^1\be_K(x,t)^2\frac{dt}t \sim \sum_{Q\in\del(x)}\be_K(Q)^2.
\end{equation}

We will use the following fact, due to the second author (see \cite{jo1990} for a proof, \cite{oki} for a generalization in $R^d$). If $\ga\ss Q_0$ is a rectifiable curve, then
\begin{equation}
\label{equTSP}
\beta_\infty(\ga):=\sum_{Q\in\del}\be_\ga(Q)^2|Q| \apprle \H(\ga).
\end{equation}

\begin{prop}\label{key}
Let $\eps\in[0,1)$ and $L>1$. There exists $M>0$ such that every curve $\ga\ss\C$  joining $0$ to $1$ such that
$$\be_\infty^\ga(x) \geq M$$
for all $x\in\ga\sm E$,  $\H(E)\leq \eps$, is of the  length
$$\H(\ga) > L.$$
\end{prop}

\begin{proof}
By the inequality (\ref{equTSP}), it is sufficient to show that $\be_\infty(\ga)$ is large enough.

We may assume that $\ga\ss Q_0=[0,1]\times[-\frac12,\frac12]$, otherwise the proof is analogous in $[-L,L+1]\times[-L,L+1]$.

Let $G_n$ be the set of maximal squares in $$\{Q\in\del,\,|Q|\geq 2^{-n} \;|\; \H(\ga\cap Q) > \frac {2L}{1-\eps}|Q|\}.$$
The squares in $G_n$ are disjoint. If $|G_n|>\frac{1-\eps}2$ then $\H(\ga) > L$.

We assume that $\beta_\infty(\ga)$ is bounded and that $|G_n|\leq \frac{1-\eps}2$ for all $n\geq 1$, $M>0$, and prove the proposition by contradiction. Let 
$$K_n=\bigcup_{Q\in G_n}Q.$$
$K_n$ is an increasing sequence. Let $K=\cup_{n\geq 1}K_n$.

For $Q\in\del$, let 
$$S(Q)=\sum_{Q\se Q'\in\del}\be_\ga(Q')^2.$$
Observe that $S(\cdot)$ is decreasing. 

By formula (\ref{equIS}), there is $c>0$ such that for any $x\in\ga\sm(K\cup E)$, there is $x\in Q\in \del$ with
$S(Q)> cM$. Let $\B$ the cover of $\ga\sm(K\cup E)$ with such maximal squares. The squares in $\B$ are disjoint and are not contained in $K$. 

Let us consider $\pi$ the projection on the real line. $\H(\pi(K))\leq\frac{1-\eps}2$ and $\H(\pi(E))\leq\eps$. Therefore
$$|\B|\geq\frac{1-\eps}2.$$ 
Let $S_n=\{Q\in\del_n \;|\; S(Q)>cM \text{ and } Q\cap\ga\sm(K\cup E)\neq\es \}$. As $S(\cdot)$ is decreasing, we may also conclude that for $n$ large enough
\begin{equation}
\label{equSn}
\#S_n\geq 2^{n-2}(1-\eps).
\end{equation}
Observe that squares in $S_n$ are disjoint from $K_n$.

\begin{lem}
\label{lemLong}
Let $Q\in\del$. Denote by $\#_k(Q)$ the cardinal of the set of squares in $\del_k(Q)$ which intersect $\ga$. Then
$$\H(\ga\cap Q)\apprge (2^{-k}\#_k(Q)-4)|Q|.$$
\end{lem}

\begin{proof}
Let $Q'\in\del_k(Q)$ intersect $\ga$ and without common boundary with $Q$. Then $\H(\ga\cap 3Q')\geq 2^{-k}|Q|$, as $\ga$ connects $Q'$ to $\pa (3Q')$. We may extract a finite (universal) number of collections of such squares $Q'$ such that in each collection, the squares $3Q'$ have disjoint interior. As we ignore $4\cdot 2^k-4$ squares in $\del_k(Q)$ with common boundary with $Q$, the conclusion follows by considering the collection of squares with maximal cardinal.
\end{proof}

As a consequence of the previous lemma, for any $Q\in\del_n$ not contained in $K_n$, for all $k\geq 1$ we have
$$\#_k(Q)\apprle 2^k.$$
As squares $Q'\in S_n$ are disjoint from $K_n$, for any $Q\in\del$ containing $Q'$, we have
\begin{equation}
\label{equCount}
\#\{Q''\in S_n \;|\; Q''\ss Q\}\apprle \frac L{1-\eps}\frac{|Q|}{|Q'|}.
\end{equation}

We may begin estimates. By the inequality (\ref{equSn}), for $n$ large enough, we have
$$2^{-n}\sum_{Q\in S_n}S(Q)\geq 2^{-n}2^{n-2}(1-\eps)cM\apprge M.$$
We have assumed that $\be_\infty(\ga)$ is bounded, so for some $C>0$
$$C\geq \sum_{Q\in\del\atop |Q|\geq 2^{-n}}\be_\ga(Q)^2|Q|\geq \sum_{Q\in\del\atop \exists Q'\se Q,Q'\in S_n}\be_\ga(Q)^2|Q|$$
$$=\sum_{k=0}^n2^{-k}\sum_{Q\in\del_k\atop \exists Q'\se Q,Q'\in S_n}\be_\ga(Q)^2
=2^{-n}\sum_{k=0}^n2^{n-k}\sum_{Q\in\del_k\atop \exists Q'\se Q,Q'\in S_n}\be_\ga(Q)^2$$
$$=2^{-n}\sum_{k=0}^n\sum_{Q\in\del_k\atop \exists Q'\se Q,Q'\in S_n}\frac{|Q|}{|Q'|}\be_\ga(Q)^2
= 2^{-n}\sum_{Q\in\del\atop \exists Q'\se Q,Q'\in S_n}\frac{|Q|}{|Q'|}\be_\ga(Q)^2$$
$$\apprge 2^{-n} \frac {1-\eps}L\sum_{Q\in \del}\#\{Q''\in S_n \;|\; Q''\ss Q\}\be_\ga(Q)^2 \apprge 2^{-n}\sum_{Q'\in S_n}S(Q')\apprge M,$$
a contradiction.
\end{proof}

\begin{remark}
{\rm Using the same notations and proof, the conclusion of Lemma \ref{lemLong} could be restated in $\R^d$ as follows
$$\H(\ga\cap (1+2^{-k+1})Q)\apprge 2^{-k}\#_k(Q)|Q|.$$
In the proof of Proposition~\ref{key} we could define 
$$K_n=\bigcup_{Q\in G_n}2Q.$$
Observe that there is a sub-collection $G'_n$ of $G_n$ such that $\pi(2G'_n)$ covers $\pi(\ga\cap K_n)$ and each point is covered at most twice. Note also that $\ga\cap K_n$ is increasing. Replacing the constant in the definition of $G_n$ by $\frac {8L}{1-\eps}$, the same proof shows that Proposition \ref{key} generalizes to curves in $R^d$, while $M$ is independent of the dimension $d$.
}
\end{remark}

Let  $\pi$ be an orthogonal  projection on the real axis in $\C$.
We have used the fact that our set is a curve only in two instances: to obtain that $\beta_\infty(\ga) \apprle \H(\ga)$ and to have $[0,1]\se\pi(\ga)$. Also, it is enough that $\H(\pi(E))\leq \eps$. We could therefore relax the hypothesis using the following well known result (see 
for example \cite{pajo}).

\begin{theo}
There exists a universal constant $C>1$ such that any continuum $K\se\R^n$ with $\H(K)<+\infty$ is contained in a curve $\ga$ such that
$$\H(K)\leq \H(\ga)\leq C \H(K).$$
\label{thmInc}
\end{theo}

We obtain the following corollary.
\begin{coro}
Let $\eps\in[0,1)$ and $L>1$. There exists $M>0$ such that every compact set $K\se[0,1]\times[-\frac12,\frac12]$ that satisfies
the following conditions,
\begin{enumerate}
\item $K\cup\{0,1\}\times[-\frac12,\frac12]$ is connected, 
\item $\be_\infty^K(x) \geq M$
for all $x\in K\sm E$,
\item $\H(\pi(E))\leq \eps$, 
\end{enumerate}
is of the length
$$\H(K) > L.$$
\end{coro}

For a set $A\se\C$, let $||A||:=\{|x| : x\in A\}$.
We will need the following version of the previous corollary.
\begin{coro}
\label{coroDisk}
Let $\eps\in[0,1)$ and $L>1$. There exists $M>0$ such that every compact set $K\se\D(0,1)$ with the property that $0\in K$, $K\cup \pa\D(0,1)$ is connected, and
$$\be_\infty^K(x) \geq M$$
for all $x\in K\sm E$, $\H(||E||)\leq \eps$, satisfies
$$\H(K) > L.$$
\end{coro}
\begin{proof}
We first unfold the annulus $\overline{\mathbb A(\frac\eps2,1)}$ to the rectangle $[\frac\eps2,1]\times[-\frac12,\frac12]$ and then map it linearly to $Q_0=[0,1]\times[-\frac12,\frac12]$. The distortion and the dilatation of this map $\varphi$ is bounded by a constant which depends only on $\eps$. Therefore $\be_K(x,t)\sim\be_{\varphi(K)}(\varphi(x),|\varphi'(x)|t)$. We obtain that 
$$\be_\infty^{\varphi(K)}(x)>C(\eps) M$$ 
for all $x\in\varphi(K)\sm \varphi(E)$ with $\H(\pi(\varphi(E)))$ close to $0$ in terms
of $\epsilon$ only. 
All other hypothesis of the previous corollary are easy to check.
\end{proof}

\begin{defi}
A set $E\ss\R^d$ with $d\geq 2$ is \emph{Ahlfors regular} if there exists $C>1$ such that for all $x\in E$ and $0<R<\diam E$,
$$C^{-1}R\leq \H(E\cap B(x,R))\leq CR.$$
\end{defi}

We will use the following result due to the second author and  Bishop (Theorem 1 in \cite{bijo}).

\begin{theo}
\label{thmBJ97}
There exists $C>0$ such that if $K\ss\R^d$ with $d\geq 2$ is a compact set of diameter $1$ and if for all $x\in K$,
$$\be_\infty^K(x)\leq M,$$
then $K$ lies on a rectifiable curve $\Ga$ of length at most $Ce^{CM}$ and which is Ahlfors regular with constant depending only on $M$.
\end{theo}

The fact that $\Ga$ is Ahlfors regular is not stated in \cite{bijo} but is an immediate consequence of the bound for the length.
\begin{proof}
As a limit of rectifiable curves with uniform bounded length containing $K$, we may assume that $\Ga$ realizes the infimum of the length of such curves. We consider a square $Q$ centered at $x\in K$ and apply the theorem to the set $|Q|^{-1}(\ol Q\cap K)$. We obtain a curve $\Ga'$ of length at most $Ce^{CM}$. We connect the endpoints of $|Q|\Ga'$ to $\pa Q$ and take the union with $\pa Q$ to obtain $\ga'$. By the choice of $\Ga$ and our construction,
$$\H(\Ga\cap Q) \leq \H(\ga')\leq |Q|(5 + Ce^{CM}).$$
\end{proof}

\begin{proof}[Proof of Theorem~\ref{thmMainTech}]
We recall that $F$ denotes the subset of $K$ of points that are not wiggly. 
Take a Borel set $E\ss K$ containing $F$ with $\H(E)<+\infty$. 
Assume that $\H(F)>0$ and $\H(K\sm F) > 0$. 
%Suppose that $\H(K) < +\infty$.
The density theorem for $\H$
(see Theorem~6.2 in \cite{matilla}) implies  that for almost all points  $x\in K\setminus E$ with respect to $\H$,
$$\limsup_{r\rightarrow 0} \frac{\H(B(x,r)\cap E)}{2r} =0\;. $$
Therefore, there is a ball $B=\D(x,R)$, $x\in K$, such that $$\H(E\cap B) < \eps R.$$
We will construct a  measure $\mu$ with the density properties as claimed in Theorem~\ref{thmMainTech}, supported on $B\cap K$,  such that $\mu(E)=0$. Let $B_0:=B$ and $\eps < \frac{1}{100}$.

The proof has three steps. 
\begin{enumerate}
\item For any ball $B=\D(x_B,R_B)$, $x_B\in K$ and such that $\H(B\cap E)<\frac{R_B}{100}$ we construct a measure $\mu_B$.
\item We construct the probability measure $\mu$ on $B_0$.
\item We show that $\mu$ has the desired scaling properties and that $\mu(E)=0$.
\end{enumerate}

\textbf{Step 1.} Let $B=\D(x_B, R_B)$ with $R_B\leq 1$, $x_B\in K$, such that $\H(B\cap E)<\eps R_B$ (with $0<\eps\leq \frac 1{100}$). Let $K_B:=K\cap\ol B$ and observe that $K_B\cup\pa B$ is a connected compact. For any $x\in K_B$, let
\newcommand{\intbes}[3]{\ensuremath{\int_{#1}^{#2}\be^2_{#3}(x,t)\frac{dt}t}}
\newcommand{\intbe}[2]{\ensuremath{\intbes{#1}{#2}{K_B}}}
$$t_B(x):=\inf\left\{r\in (0, R_B) : \intbe r{R_B} \leq M \right\},$$
where $M$ is a large constant that will be specified later.
Observe that $t_B(x)=0$ if and only if $\intbe 0{R_B}\leq M$. Let
$$Z(B):=\ol{\{ x\in K_B : t_B(x)=0\}}.$$
One can easily check that for any $x\in Z(B)$,
$$\intbes 0{R_B}{Z(B)} \leq \intbe 0{R_B} \leq 4M,$$
therefore $Z(B)\se F \se E$. We obtain that $\H(Z(B)) < \eps R_B$.

We show that there exists $C_M>0$, that depends only on $M$, and a measure $\mu_B$ such that
\begin{equation}
\label{equMuB}
\begin{array}{ll}
\text{1.} & R_B \leq \mu_B(\C) = \mu_B(B) \leq C_M R_B, \vspace{1mm}\\
\text{2.} & \mu_B(\D(x,r)) \leq C_M r, \text{ for all } x\in\C \text{  and } r>0.
\end{array}
\end{equation}

Set $W(B):=K_B \sm Z(B)$. By the standard covering lemma (see \cite{hein}, page 2), there exists a countable set $X'(B)\ss W(B)$ such that 
$$W(B)\se \bigcup_{x\in X'(B)}\D(x,10t_B(x)),$$
and the balls $\D(x, 2t_B(x))$, $x\in X'(B),$ are pairwise disjoint. 

From now on we make a standing assumption (it is enough to take $M\geq 4\log 10$) that for all $x\in X'(B)$,
\begin{equation}
\label{equSR}
t_B(x)\leq 10^{-4}R_B.
\end{equation}

The following lemma provides a key estimate which allows to distribute the measure $\mu$ on the balls $\D(x,t_B(x))$ centered at $X(B)$ and prove the scaling properties of $\mu$ in Step~3.
\begin{lem}
\label{lemTen}
If $M$ is sufficiently large, then there exists $X(B)\se X'(B)$ such that for each $x\in X(B)$, $\H(\D(x,t_B(x))\cap E)<\frac\eps 2\, t_B(x)$,
$$\sum_{x\in X(B)}t_B(x) \geq 10 R_B,$$
and
$$\sum_{x\in X(B)}\H(\D(x,t_B(x))\cap E) \leq \frac{\eps}2 R_B.$$
\end{lem}

It is not hard to check that the set $Y=Z(B)\cup X'(B)$ is compact. Also, one can check that for every  $x\in Y$, $\intbes 0{R_B}{Y} \leq 100 M$. By Theorem~\ref{thmBJ97}, the set $Y$ is contained in an Ahlfors regular curve $\Ga_B$ with the length comparable to $R_B$ (and the constant depending only on $M$). As the balls $\D(x, 2t_B(x))$  from the the same
$\mathcal G_i(B)$ are pairwise disjoint, we may assume that $\Ga_B$ contains a cross 
$$G(x):=[x-t_B(x), x+t_B(x)]\cup[x - it_B(x), x + it_B(X)]$$ 
for every  $x\in X'(B)$. We define 
$$\mu_B:=\H_{|G_B}, \text{ where }G_B:=\bigcup_{x\in X(B)} G(x).$$
By the properties of $\Ga_B$ it is easy to check that $\mu_B$ satisfies the inequalities (\ref{equMuB}).

The following proof concludes Step 1.
\begin{proof}[Proof of Lemma \ref{lemTen}]
We use a geometric construction to show that if $M$ is sufficiently large, then
\begin{equation}
\label{equThirty}
\sum_{x\in X'(B)}t_B(x) \geq 23 R_B,
\end{equation}

As $t_B(x)\leq 10^{-4}R_B$ for all $x\in X'(B)$, there is a partition of $X'(B)=X_1(B)\cup X_2(B)$ such that for each $i\in\{1,2\}$,
$$\sum_{x\in X_i(B)}t_B(x) \geq 11 R_B.$$
As $\H(B\cap E)<\eps R_B$, there is $i_0\in\{1,2\}$ such that 
$$\sum_{x\in X_{i_0}(B)}\H(\D(x,t_B(x))\cap E) \leq \frac{\eps}2 R_B.$$

Let $X(B)=\{x\in X_{i_0}(B)\ :\ \H(\D(x,t_B(x))\cap E)<\frac\eps 2\, t_B(x)\}$ and suppose that
$$\sum_{x\in X_{i_0}(B)\sm X(B)}t_B(x) \geq R_B.$$
Then 
$$\sum_{x\in X_{i_0}(B)\sm X(B)}\H(\D(x,t_B(x))\cap E) \geq \frac{\eps}{2}\, R_B,$$
which contradicts the choice of $i_0$, as balls $\D(x,t_B(x))$ with $x\in X_{i_0}(B)$ are disjoint.

In the sequel, we prove the inequality \requ{Thirty}. For every $x\in X'(B)$, let
$$H(x):=\pa\D(x,2t_B(x))\cup\pa\D(x,10t_B(x))\cup [x-10t_B(x), x+10t_B(x)].$$
Then $H(x)$ is connected and $\H(H(x)) \leq 100 t_B(x)$.
\def\xxb{{x\in X'(B)}}
Let 
$$Z'(B):=Z(B)\sm\bigcup_\xxb \D(x,10t_B(x)).$$ 
Observe that $K_B\ss Z'(B)\cup\bigcup_{\xxb} \D(x,10t_B(x))$. Let us  define
$$S(B):=Z'(B)\cup\bigcup_\xxb H(x).$$
As $\H(Z'(B)) \leq 10^{-2}R_B$, it is enough to show that $\H(S(B)) \geq 2301 R_B$ in order to prove the inequality \requ{Thirty}.
We may assume that for any $x\neq x'\in X'(B)$, $\D(x,10t_B(x)) \not\ss \D(x',10t_B(x'))$. The reader can therefore check that $S(B)$ is compact and $S(B)\cup\pa B$ is a connected compact.
We will show that for every $x\in S(B)\sm Z'(B)$,
\begin{equation}
\label{equTrans}
\be_\infty^{S(B)}(x) \geq 10^{-5}M,
\end{equation}
provided $M$ is large enough. Having established  the estimate~(\ref{equTrans}), we can conclude the proof of the inequality ~\requ{Thirty} by applying Corollary \ref{coroDisk} to $S(B)$ with $Z'(B)$ as an exceptional set, $\eps=10^{-2},L=2301$ and $M$ large enough.

We still have to prove the inequality \requ{Trans}. For this, let $y\in S(B)\sm Z'(B)$. We can find $x_0\in X'(B)$ such that 
$$t_B(x_0)=\sup\{t_B(x) : x\in X'(B), |x-y|<20t_B(x)\}.$$
We will show that for $t \geq 10^3t_B(x_0)$,
\begin{equation}
\label{equBeTr}
\be_{S(B)}(y,t) \geq 10^{-2}\be_{K_B}(x_0, 10^{-2}t).
\end{equation}
To this end, observe that by the choice of $x_0$ and $t$, we have $\D(x_0,10^{-2}t)\ss \D(y,t)$, and therefore
$$\be_{K_B}(y,t) \geq 10^{-2}\be_{K_B}(x_0, 10^{-2}t).$$ 
It is enough to prove that 
\begin{equation}\label{equ:bl}
\be_{S(B)}(y,t) \geq \be_{K_B}(y,t).
\end{equation}
Let $L$ be a line minimizing $\be_{S(B)}(y,t)$ and take a point  $z \in \D(y,t)\cap (K_B \sm S(B))$. 
In particular, $z \in K_B \sm Z'(B)$ so there is $x\in X'(B)$ such that $z\in \D(x,10t_B(x))$. By the choice of $x_0$,   
$\D(y,t)\not\ss \D(x, 10t_B(x))$.  As $\pa\D(x, 10t_B(x))\ss S(B)$, we have that $\pa\D(x, 10t_B(x)) \cap \D(y,t)$ is contained in a $t\be_{S(B)}(y,t)$ neighborhood of $L$, and so is $z$. The inequality (\ref{equ:bl}) is now established.

For each $y\in S(B)\sm Z'(B)$, we integrate the inequality \requ{BeTr} and obtain
$$
\begin{array}{rcl}
\int_{10^3t_B(x_0)}^{R_B}\be^2_{Y(B)}(y,t)\frac{dt}t & \geq & 10^{-4}\int_{10^3t_B(x_0)}^{R_B}\be^2_{K_B}(x_0,t)\frac{dt}t\\
& \geq & 10^{-4}M-3\log 10\\
& \geq & 10^{-5}M,
\end{array}
$$
if $M$ is large enough. This proves the inequality \requ{Trans}.
\end{proof}

\textbf{Step 2.} We define a sequence $(\mu_n)_{n\geq 0}$ of probability measures supported respectively on a decreasing sequence of compact neighborhoods of $K$, having $K$ as their intersection. The measure $\mu$ is then a weak limit of this sequence. 

For each $n\geq 0$, we define a collection of disjoint balls $\F_n$ as follows. Let $\F_0:=\{B_0\}$ and 
$$\F_{n+1}:=\{\D(x,t_B(x)) : x\in X(B), B\in \F_n\}.$$
Let $B_1,B_2\in\F_n$, $B_1',B_2'\in\F_{n+1}$ such that $B_1'=\D(x,t_{B_1}(x))$ with $x\in X(B_1)$ and respectively $B_2'=\D(y,t_{B_2}(y))$ with $y\in X(B_2)$. As $2B_1$ and $2B_2$ are disjoint and $t_B(x)\leq 10^{-3}R_{B_1}$, $t_B(y)\leq 10^{-3}R_{B_2}$ we obtain  that
$$B_1'\cap B_2'=\es, B_1\cap B_2'=\es \text{ and } B_1'\cap B_2=\es.$$

We will use the measures $\mu_B$ constructed at the previous step to define inductively the sequence $(\mu_n)_{n\geq 0}$. Let
$$\mu_0:=\frac{\mu_{B_0}}{\mu_{B_0}(\C)}.$$
Assume that $\mu_0,\ldots,\mu_n$ have been defined with the following properties.
\begin{description}
\item{(P1)} $\spt\mu_n$ is contained in the disjoint union of balls in $\F_n$
which is contained in a $10^{-n}$-neighborhood of $K$. Also $\H(E\cap \spt \mu_n) < \frac{2^{-n}}{100} R$ by Lemma \ref{lemTen}, where $R$ is the radius of $B_0$.
\item{(P2)} $\mu_n(B) \leq 10^{-k}R^{-1}R_B$ if $B\in\F_k$ for some $k\leq n$ ($R_B$ is the radius of $B$).
\end{description}
We define $\mu_{n+1}$ in the following way. For any ball $B\in\F_{n+1}$, 
\begin{equation}\label{equ:def}
\mu_{n+1|B}:=\frac{\mu_n(B)}{\mu_B(B)}\mu_B.
\end{equation}

Observe that $\mu_{n+1}(B)=\mu_n(B)$ and that by Lemma \ref{lemTen}, for every $x\in X(B)$ we have
$$\frac{\mu_{n+1}(\D(x, t_B(x)))}{\mu_{n+1}(B)}=\frac{t_B(x)}{\sum_{y\in X(B)}t_B(y)}\leq\frac{t_B(x)}{10 R_B}.$$

Properties (P1) and (P2) for $\mu_{n+1}$ are direct consequences of the inequality \requ{SR} and Lemma \ref{lemTen} and the choice of $X(B)\se X'(B)$ for any ball $B$.

\textbf{Step 3.} Observe that if $B\in\F_n$ is a ball of radius $R_B$, then 
\begin{equation}
\label{equSBou}
\mu(B)\leq 10^{-n}R^{-1}R_B.
\end{equation}
We want to show that for every $x\in K$ and $r>0$,
\begin{equation}
\label{equBou}
\mu(\D(x,r))\leq CR^{-1}r\exp\left(-C'\intbes r1K\right),
\end{equation}
where $C,C'>0$ are universal constants. 

Let us note that it is enough to prove this bound for $x\in\spt\mu$. Otherwise, we have either $\spt\mu\cap \D(x,r)=\es$, so $\mu(\D(x,r))=0$, or for some $y\in \spt\mu\cap \D(x,r)$, $\D(x,r)\ss\D(y,2r)$, so
$$\mu(\D(x,r))\leq \mu(\D(y,2r))\leq 2CR^{-1}r\exp\left(-C'\int_{2r}^{1}\be_{K}(y,t)^2\frac{dt}t\right).$$
As $|x-y|<r$, a simple computation shows that 
$$1 + \intbes r1K \sim 1 + \int_{2r}^{1}\be_{K}(y,t)^2\frac{dt}t,$$ 
with  universal constants.

Fix $x\in\spt\mu$ and $r>0$. If $r\apprge R$ there is nothing to prove. We have $\{x\}=\cap_{n\geq 0}B_n$, where $B_n\in\F_n$ for all $n>0$. Let 
$$N:=\max\{n : \D(x,r)\se 2B_n\}.$$
A direct  computation leads to  
\begin{equation}
\label{equNM}
\intbes r1K \apprle (N+4)M.
\end{equation}
On the other hand, 
$$\mu(2B_N) \leq 10^{-N}R^{-1}R_{B_N}.$$
If $r\sim R_{B_N}$, the last two inequalities and the properties \requ{MuB} imply the conclusion \requ{Bou}.

Suppose that $r<10^{-6}R_{B_N}$. For any $B\in\F_{N+1}$ which intersects $\D(x,r)$, we have $R_B \leq 4r$. Otherwise $\D(x,r)\se 2B$ which contradicts the choice of $N$. Thus $B\se\D(x,10r)$. Denote $\mathcal K:=\{B\in\F_{N+1} : B\cap\D(x,r)\neq\es\}$. Using the inequality \requ{SBou} and the fact that $\Ga_{B_N}$ constructed in Step 1 is Ahlfors regular, we can estimate
$$
\begin{array}{rcl}
\mu(\D(x,r)) & \leq & \sum_{B\in\mathcal K}\mu(B) \leq 10^{-N-1}R^{-1}\sum_{B\in\mathcal K}R_B \\
& \leq & c10^{-N-1}R^{-1}\H(\Ga_{B_N}\cap \D(x,10r))\\
& \leq & c'10^{-N}R^{-1}r,
\end{array}
$$
where $c,c'>0$ are universal constants. Combined with the inequality \requ{NM}, this proves the conclusion \requ{Bou}.

By (P1), $\H(E\cap \spt\mu)=0$. The inequality~\requ{Bou} implies that 
$$\mu(\D(x,r))\leq CR^{-1}r$$  
for every ball $\D(x,r)$, $x\in K$, $r\leq R$ and thus the measure $\mu$ vanishes on $E\supset F$.
\end{proof}

\subsection{Proof of Theorem~\ref{thmMain}}
For simplicity, suppose that $\diam{K}=1$. Theorem~\ref{thmMainTech} supplies  a Radon  probability measure $\mu$ with $\mu(W)=1$.

%We will prove first the second claim of Theorem~\ref{thmMain}.
For every $x\in W$ we define a measurable function 
% two non-negative functions, $r(x), 
$\beta(x): W \mapsto [0, +\infty)$ by
%by the following  conditions,
$$\liminf_{r\ra 0}\frac{\int_r^{1}\be_K^2(x,t)\frac{dt}t}{-\log r}= \beta^2(x)\:.$$
%and $r(x)$ is the largest number such that for every $r\leq r(x)$
%$$\int_r^{1}\be_K^2(x,t)\frac{dt}{t}\geq -\frac{1}{2}\omega(x)\log r\;.$$
%The functions $r(x)$ and $\omega(x)$ are $\mu$ measurable and thus there exist
%$\omega'>0$ and $\diam K > r'>0$ such that the set
%$$W'= \{x\in W:  r(x)\geq r'\;\;\mbox{and}\:\: \beta(x)\geq \beta'\}$$
%is of positive $\mu$ measure.  
For every $x\in W$ we have that 
\begin{eqnarray*}
\liminf_{r\ra 0} \frac{\log \mu(\D(x,r)}{\log r}&\geq & \liminf_{r\ra 0}\frac{-\log r + C' \int_r^{\diam \!K}\be_K^2(x,t)\frac{dt}t}{-\log r}\\
&= & 1+ C' \liminf_{r\ra 0}\frac{\int_r^{1}\be_K^2(x,t)\frac{dt}t}{-\log r}~ = ~ 1+C'\beta^2(x)\;. 
\end{eqnarray*}
The mass distribution principle implies that 
$$\HD(K)\geq \HD(W')\geq 1+ C' \mbox{essup}_{\mu}\beta^2(x)\;,$$
where $\mbox{essup}_\mu\beta^2(x)$ is an essential supremum of $\beta^2(x)$ with respect to $\mu$.
%We take an arbitrary   collection of balls $B$ centered at $W'$
%with  $\diam B\leq r'\leq \diam K$.  By Besicovitch covering lemma, we can choose
%at most $M$  countable subcollections of balls $\{B_i\}$, $0<i\leq M$,  
%such that  the union of all the balls from these $M$ subcollections contains $W'$
%and every two balls from a given subcollection
%are pairwise disjoint.  Denote $d_B:=\diam B$.  We use Theorem~\ref{thmTech}
%and  the definition  of $\omega'$ and $r'$ to obtain that 
%and the fact that $\mu(W')>0$,  to obtain that
%\begin{eqnarray*}
%\mu(W')&\leq& M \sum_B \mu(B) \leq MCR^{-1} \sum_B d_B\exp \!\left(-C'\int_{d_B}^{\diam \!K}\be_K^2(x,t)\frac{dt}t\right)\\
%&\leq& MCR^{-1} \sum _B d_B\exp\!\left( C'\omega' \log {d_B}\right) = MCR^{-1}\sum _B d_B^{1+C'\omega'}\,.
%\end{eqnarray*}
%Since $\mu(W')>0$,   $\HD(K)\geq \HD(W')\geq 1+C'\omega'$.

The proof of the theorem follows as $\mbox{essup}_{\mu} \beta(x) >0$ if $\beta(x)>0$ and 
$\mbox{essup}_{\mu} \beta(x) \geq \beta^2_0$ if $\beta(x)\geq \beta_0$ for every $x\in W$.
\medskip
%\paragraph{Connectivity.}
%Corollary~\ref{coro:weak1} allows to  relax  the hypothesis about connectivity in Theorem~\ref{thmMain}.
%\begin{coro}\label{coro:weak}
%The claim of Theorem~\ref{thmMain} holds if the  hypothesis about connectivity of $K$ is replaced by the condition ${\rm(i)}$ of Corollary~\ref{coro:weak1}. 
%Moreover, there exists $W'\subset W$, $\H(W\setminus W')=0$, such that 
%$$\HD(K)\geq 1+c\beta_{\mbox{ess}}^2\,, $$
%where
%$$\beta_{\mbox{ess}} = \sup_{z\in W'} \inf_{x\in \D(z,r(z))\cap W'} 
%\liminf_{r\ra 0}\frac{\int_r^{1}\be_K^2(x,t)\frac{dt}t}{-\log r}\:.$$

% there is a constant $\omega >0$ such that for every $z\in W$ there is  $r>0$ such that every connected component of $B(z,r)\cap K$ has the diamater bigger
%than $\omega r$.  
%\end{coro}

\subsection{Proof of Theorem~\ref{theo:nonpor}}

As the constant $C$ may be large, it is enough to prove the dimension estimate when $\eps$ is asymptotically close to $0$. We will construct a measure on $K$ following the inductive strategy of the proof of Theorem \ref{thmMainTech}. Let $M$ be defined as in the proof of Theorem \ref{thmMainTech}. We assume $\eps \ll \la^2 e^{-2M}$ and define
$$\eps'=\frac{2e^{2M}\eps}{\la^2}.$$
If $K$ is $\eps'$-porous in $2B$, we define $\mu_B$ as in Step 1. of the proof of Theorem \ref{thmMainTech}. Otherwise, we set $t_B(x)=\eps'R_B$ for all $x\in K_B$. $X(B)$ has the property that for all $x,x'\in X(B)$, $\H(E\cap B(x,\eps'R_B))<\frac{\eps'R_B}{100}$, the balls $B(x,2\eps'R_B)$ and $B(x',2\eps'R_B)$ are disjoint, $\frac{10}{11}|X'(B)| \leq |X(B)|$, and
$$K\se\bigcup_{x\in X'(B)}B(x, 10\eps'R_B).$$
As $K$ is not $\eps'$-porous in $2B$, we obtain that
$$B\se\bigcup_{x\in X'(B)}B(x, 11\eps'R_B).$$
As balls $B(x,2\eps'R_B)$ with $x\in X(B)$ are disjoint, using a volume argument, we obtain
$$10\cdot 11^{-d-1}                                                                                                                                                                                                                                                                                                                                                                                                                                                                                                                                                                                                                                                                                                                                                                                                                           \eps'^{-d} < |X(B)| < 2^{-d}\eps'^{-d}.$$
We define the measure $\mu_B$ in the same way as in the proof of Theorem \ref{thmMainTech}. We observe that for $x\in X(B)$, 
$$\frac{\mu_B(B(x, t_B(x)))}{\mu_B(B)}=\frac{t_B(x)}{\sum_{y\in X(B)}t_B(y)}\leq\frac{\eps'R_B}{\eps'R_B|X(B)|}<\frac{11^{d+1}}{10}\frac{t_B(x)}{R_B}\eps'^{d-1}.$$
We obtain a new form of the inequality \requ{SBou}. For $B\in\F_n$, let $n=k_1+k_2$, where $k_2$ is the number of steps $m$ at which $K$ is not $\eps'$-porous in $2B_m$, where $B_m\in\F_m$ with $B\se B_m$. By the previous inequality and Lemma \ref{lemTen} we obtain
$$\mu(B)\leq 10^{-k_1}\left(\frac{11^{d+1}}{10}\eps'^{d-1}\right)^{k_2}R^{-1} R_B.$$

The new measure $\mu_B$ that we have constructed does not satisfy the inequalities (\ref{equMuB}), because the constant $C_M$ is replaced by $C(d,\eps')=2^{-d}\eps'^{1-d}$. Repeating the argument from Step 3. of Theorem \ref{thmMainTech}, we obtain for all $x\in K$ and $0<r<\diam K$
\begin{equation}
\label{equTE}
\mu(B(x,r))\leq  C' \eps'^{-d}\left(11^{d+1}\eps'^{d-1}\right)^{k_2}R^{-1}r,
\end{equation}
where $C'>0$ is a universal constant. 

Observe that as a consequence of (P1), we obtain that $\H(\spt\mu\cap E)=0$ therefore $\mu(E)=0$ and $\mu(W)=1$. We now fix $x\in\spt\mu\cap W$, $\de>0$ and $r=\la^N$ for some large $N\in\N^*$ such that 
$$\#P_\la(N) > N\left(\kappa-\frac\de 2\right),$$
where 
$$P_\la(N)=\{m\in\{[-\log R] + 1,\ldots,N\}\ :\ K \text{ is not } \eps \text{-porous in }B(x,\la^m)\},$$
where we denote by $[x]$ the integer part of $x$, and $R$ is the radius of $B_0$.

For each $n\geq 0$, let $B_n\in\F_n$ containing $x$ and $R_n=R_{B_n}$. By construction, $\{x\}=\cap_{n>0}B_n$. Let $n_0=n_0(r)$ be the first $n>0$ such that $2B_n\se B(x,r)$. We construct a map 
$$\chi:P_\la(N)\ra P_\F(n_0)$$
that is at most $s$ to $1$, where $s=\frac{\log \eps}{\log \la}$ and
$$P_\F(n_0)=\{n\in\{0,\ldots,n_0\}\ :\ K \text{ is not } \eps' \text{-porous in }2B_n\}.$$ 

As $k_2=\#P_\F(n_0)$, if $N$ is large enough, we obtain 
\begin{equation}
\label{equK2}
k_2 \geq \frac Ns(\kappa-\de).
\end{equation}

Let $m\in P_\la(N)$ and $n=n_0(\la^m)-1$. We define 
$$\chi(m)=\left\{
\begin{array}{ll}
n & \text{ if }n\in P_\F(n_0),\\
n+1 & \text{ otherwise.}
\end{array}
\right.$$
If $n\in P_\F(n_0)$ then $R_{n+1}=\eps'R_n$, thus there are at most $\left[\frac{\log \eps'-\log 2}{\log \la}\right]+1$ values in $P_\la(N)$ mapped to $n$ by the first branch of $f$. If $n\notin P_\F(n_0)$, then by construction
$$M=\int_{R_{n+1}}^{R_n}\be_K(x,t)^2\frac{dt}t\geq \int_{R_{n+1}}^{\la^m/4}\be_K(x,t)^2\frac{dt}t.$$
Because $K$ is not $\eps$-porous in $B(x,\la^m)$, it is not $\eps'$-porous in $B(x,t)$ for every $t\in[e^{-2M}\la^m,\la^m]$. Therefore 
$$\be_K(x,t)\geq (1-\eps')\text{, for all }t\in[e^{-2M}\la^m,\la^m].$$
We can therefore estimate
$$\int_{e^{-2M}\la^m}^{\la^m/4}\be_K(x,t)^2\frac{dt}t \geq (1-\eps')^2(2M - \log 4) > M,$$
as $\eps'$ is small and $M$ is large. We conclude that $R_n > e^{-2M}\la^m$ so there are at most $\left[\frac{-2M}{\log \la}\right]+1$ values in $P_\la(N)$ mapped to $n+1$ by the second branch of $f$. This completes the proof that $\chi$ is at most $s$ to $1$.

Combining the inequalities \requ{TE} and \requ{K2}, we can compute that 
$$\frac{-\log \mu(B(x,r))}{-\log r} \geq \frac{C(d,\eps',R)}{N\log\la} + 1 + k_2\frac{(d+1)\log 11 + (d-1)\log \eps'}{N\log\la}$$
$$\geq \frac{C(d,\eps',R)}{N\log\la} + 1 + (d-1)(\kappa-\de)\frac{\log \eps'}{\log \eps} + \frac{(d+1)(\kappa-\de)\log 11}{\log \eps}.$$
By the definition of $\eps'$, 
$$\frac{\log \eps'}{\log \eps}=1 + \frac{\log\la -2M - \log 2}{|\log \eps|}.$$
Therefore,
$$\liminf_{N\ra\infty}\frac{\log \mu(B(x,r))}{\log r} \geq 1 + \kappa\left(d - 1 - \frac{C(d,\la)}{|\log\eps|}\right),$$
which, by the mass distribution principle, yields the desired lower bound for $\HD(K)$.
\subsection{About the hypotheses of Theorem~\ref{theo:nonpor}.}\label{sec:theo5}
A prototype  example
of  the unit circle $S^1\subset \C$ together with two  smooth curves  $\gamma_{-} $ and $\gamma_+$ winding infinitely many times around it both from inside and outside, shows that the hypothesis that $\H(E)<\infty$
can not be dropped. Indeed, assume that  a point traveling along any of these two smooth curves
is approching $S^1$  slowly enough  so that for every $\eps>0$,  $S^1$ is not $\eps$-porous at any $z\in S^{1}$ at all  scales small enough. 
Clearly, $\HD(S^1\cup \gamma_-\cup \gamma_+)=1$ and $\H(\gamma_-\cup\gamma_+)=\infty$.

The following examples show that the hypotheses about connectivity and $\H(W)>0$ are  also essential. 

{\textbf{ Example 1.}} 
We recall that if $S_i\subset \C$, $i=1,2$, then
%$$
$S_1+S_2$
%=\{v_1+v_2 : v_i\in S_1\; \mbox{and}\; v_2\in S_2\}. $$ 
stands for a set of all $z\in \C$ such that $z=v_1+v_2$, where $v_1\in S_1$ and $v_2\in S_2$.

Let $[0,1]\supset K_m=[0,2^{-4m}]+\{k2^{-2m}\ :\ k=1,\ldots, 2^{2m}-1\}$. In the formula below, $i^2=-1$. 
$$K=[0,1]\cup\bigcup_{m\geq 1}\bigcup_{k=0}^{2^{m}-1}\left(\pm i(2^{-m} + k2^{-2m}) + K_m\right).$$
One can easily check that $\H(K)<\infty$ and that for any $\eps>0$, $K$ is not $\eps$-porous at scales of density $1$ at  every $z\in(0,1)$.

{\textbf {Example 2.}} Let $\alpha\in (0,1)$. We consider a standard $\alpha$-Cantor set $K_\alpha\subset [0,1]$ defined as an invariant set for
the map $g: [0,1]\mapsto \R$, 
$$
g(x)=\left\{
\begin{array}{ccc}
x/ \alpha  & \mbox{if}  & x\in [0,1/2]  \\

(1-x)/\alpha &\mbox{if} & \; x\in [1/2,1]\; 
 %&   &   
\end{array} \right.
$$
Geometrically, $K_\alpha$ is an  intersection of 
the union of $2^n$ closed intervals $J_j^n$, $j=1,\dots 2^n$,  each of the length $\left(\frac{1-\alpha}{2}\right)^n$. Every
$J_j^n$ is a connected component of $g^{-n}[0,1]$. Let $Q_j^n$ be a square with the base equal to $J_j^n$  contained
in the  half-plane $\Im z\geq 0$. In the formula below, $i$ stands for the imaginary number. 
Put $R_j^n$ to be the union of $\partial Q_j^n \cup \{ 2J_j^n + i \frac{k}{n}|J^n_j|: k=1,\ldots, n\}$ and its reflection with respect to the
real axis. The set
$$R_\alpha := [0,1]\cup  \bigcup_{n \geq 1}\bigcup_{j=1}^{2^{n}} R_j^n$$
is connected and every point of $K_\alpha$ is not $\eps$-porous for any $\eps>0$ at any scale   small enough.
One can easily check that $\H(R_\alpha)$ is bounded by a constant which depends only on $\alpha$. Suppose that  $\alpha_n$ tends to $0$
and denote the corresponding $R_\alpha$ by $R_n$ and $K_\alpha$ by $K_n$.  The  union 
$$R= [0,1]\cup \bigcup_{k\geq 1}\left(\beta_k R_{k}+\frac{1}{k}\right)\;,$$ 
where $\beta_k$  are chosen so that $\beta_k \H(R_k)\leq 4^{-k}$ is a continuum with the property that every point $z\in \bigcup_{n\geq 1} K_n$ is not $\epsilon$-porous
for any $\eps>0$ and any  scale  small enough.  Therefore, $\HD(W)\geq \sup_{n\geq 1} \HD(K_n)=1$. Nevertheless,  $\H(R)<\infty$.

%The following examples show that the hypotheses about connectivity and $\H(W)>0$ are  essential. 

\subsection{Universal  version of $\mu$}

If we do not require that the measure $\mu$ from Theorem~\ref{thmMainTech} has the support disjoint from  $F$ then the construction
can be  modified to obtain  a uniform scaling property of $\mu$.  

\begin{theo}
\label{thmTech}
Let $K\ss\R^d$ with $d\geq 2$ be a connected compact of diameter $1$. There exist a universal constants $C'>0$, a constant $C>0$ depending only on $d$, and a Radon probability measure $\mu$ supported on $K$ such that for all $x\in K$ and $r>0$
$$\mu(B(x,r))\leq Cr\exp\left(-C'\int_r^1\be_K^2(x,t)\frac{dt}t\right).$$
\end{theo}
\begin{proof}
The proof of Theorem~\ref{thmTech} is very similar  to the proof of Theorem~\ref{thmMainTech} and is based on 
Proposition~\ref{key}
and the constructions from the proof of Theorem~45 in~\cite{pajo}, compare~\cite{guy}. 
As in the proof of Theorem~\ref{thmMainTech} we have three steps. 
\begin{enumerate}
\item For any ball B centered on $K$ we construct a measure $\mu_B$.
\item We construct the probability measure $\mu$ on $K$.
\item We show that $\mu$ has the desired scaling properties.
\end{enumerate}
The estimates from the third step are very much the same  as in the proof of Theorem~\ref{thmMainTech} 
(with $B_0$ replaced by $K$, $R$ replaced by $1$, and the claims about $\H(E\cup \spt \mu_j)$ and $\H(E\cap \spt \mu)$  dropped). The first and the second steps are slightly different as they must account for the existance of non-wiggly 
parts of $K$. We use the notation from the proof of Theorem~\ref{thmMainTech}.

\textbf{Step 1.} Let $B=\D(x_B, R_B)$ be a ball centered on $K$ with $R_B\leq 1$. Let $K_B:=K\cap\ol B$ and observe that $K_B\cup\pa B$ is a connected compact. 
\newcommand{\intbes}[3]{\ensuremath{\int_{#1}^{#2}\be^2_{#3}(x,t)\frac{dt}t}}
\newcommand{\intbe}[2]{\ensuremath{\intbes{#1}{#2}{K_B}}}
We say that $B\in \mathcal G$ (respectively that $B\in \mathcal B$) if $\H(Z(B)) \geq \frac{R_B}{100}$ (respectively if $\H(Z(B)) < \frac{R_B}{100}$).

Assume first that $B\in \mathcal G$. Since $\intbes 0{R_B}{Z(B)} \leq 4M$ for any $x\in Z(B)$, by Theorem \ref{thmBJ97}, the set $Z(B)$ is contained in an Ahlfors regular curve $\Ga_B$ whose regularity constant depends only on $M$ and of length comparable to $R_B$. We set 
$$\mu_B:=\H_{|Z(B)}.$$
By the properties of $\Ga_B$, there exists a constant $C_M>0$ that depends only on $M$ such that the conditions~$(\ref{equMuB})$
are satisfied.

If $B\in \mathcal B$ then we can repeat the construction of $\mu_B$ from the first step of Theorem~\ref{thmMainTech}.
The only modification is that we define $\mu$ on $G'_B:=\bigcup_{x\in X'(B)} G(x)$ rather than on
$G_B:=\bigcup_{x\in X(B)} G(x)$, as it was the case in the proof of the Theorem~\ref{thmMainTech}. We put
$$\mu_B:=\H_{|G'_B}$$
By the properties of $\Ga_B$ it is easy to check that $\mu_B$ satisfies the inequalities (\ref{equMuB}).

\textbf{Step 2.} We construct a sequence $(\mu_n)_{n\geq 0}$ of probability measures supported respectively on a decreasing sequence of compact neighborhoods of $K$, having $K$ as their intersection. The measure $\mu$ is a weak limit of this sequence. 

For each $n\geq 0$, we define a collection of disjoint balls $\F_n$ as follows. Let $x_0,y_0\in K$ be such that $|x_0-y_0|=1=\diam K$. Let $\F_0:=\{\D(x_0,1)\}$ and 
$$\F_{n+1}:=\{\D(x,t_B(x)) : x\in X(B), B\in \F_n\cap\mathcal B\}.$$
\textit{Remarks.} 1) If $\F_n\se\mathcal G$ then $\F_{n+1}=\es.$
\\2) Let $B_1,B_2\in\F_n$, $B_1',B_2'\in\F_{n+1}$. As in the second step of the proof of Theorem~\ref{thmMainTech}
we obtain that 
$$B_1'\cap B_2'=\es, B_1\cap B_2'=\es \text{ and } B_1'\cap B_2=\es.$$

Let $\G_n:=\F_n\cap\G$ and $\B_n:=\F_n\cap B$. We will use the measures $\mu_B$ constructed in the previous step to define inductively the sequence $(\mu_n)_{n\geq 0}$. Let
$$\mu_0:=\frac{\mu_{B}}{\mu_{B}(\C)},$$
where $B=\D(x_0,1)$.
Assume that $\mu_0,\ldots,\mu_n$ have been defined with the following properties.
\begin{description}
\item{(P1)} $\spt\mu_n$ is contained in the disjoint union of balls
$$\F_n \cup \bigcup_{k=0}^{n-1}\G_k\,,$$
which is contained in a $10^{-n}$-neighborhood of $K$.
\item{(P2)} $\mu_n(B) \leq 10^{-k}R_B$ if $B\in\F_k$ for some $k\leq n$ ($R_B$ is the radius of $B$).
\end{description}
Define $\mu_{n+1}$ in the following way. For any ball $B\in\G_k,k\leq n$, 
$$\mu_{n+1|B}:=\mu_{n|B}.$$
For any ball $B\in\F_{n+1}$, 
$$\mu_{n+1|B}:=\frac{\mu_n(B)}{\mu_B(B)}\mu_B.$$
Properties (P1) and (P2) for $\mu_{n+1}$ are direct consequences of the inequalities \requ{Thirty} and \requ{SR}, and the fact that $Z(B)\ss K$ for any ball $B$.

\end{proof}

\begin{coro}\label{coro:con} 
Let $\mu$ be supplied either by Theorem~\ref{thmMainTech} or Theorem~\ref{thmTech}. Then $\mu$ is absolutely continuous with respect to $1$-dimensional Hausdorff measure $\H$.
\end{coro}
\begin{proof}
In both cases, for each measure $\mu$, there exists a constant $M>0$ (which is universal if $\mu$ is provided by Theorem~\ref{thmTech}) such that for every ball $B(x,r)$, $x\in K$, $r<\diam K$,
\begin{equation}\label{equ:mu1}
\mu(B(x,r))<Mr\;.
\end{equation}
A standard argument shows that $\H(A)=0\Longrightarrow \mu(A)=0$.
\end{proof}

\section{Topological Collet-Eckmann rational  maps}\label{sec:TCE}
TCE property (see the definition in Section \ref{sec:bowen}) implies the so called {\em exponential shrinking} which states that there exists  a  number $\xi<1$ 
(depending  on $\delta$, $\kappa$, and $L$ but not on $z$) so that

\[\diam \mbox{Comp}_{z}f^{-n}(B_{\delta}(f^{n}(z))\leq \xi^{n}\;\]
for every  $n\in \N$ and $z\in\mathcal{J}$.

\paragraph{Proof of Theorem~\ref{TCE}.} The proof is by contradiction. 
We may assume, by decreasing slightly $\delta$,  that every
component   $\mbox{Comp}_{z}f^{-n}(B_{\delta}(f^{n}(z))$ satisfies the inclusions,
$$B(z,r_n)\subset  \mbox{Comp}_{z}f^{-n}(B_{\delta}(f^{n}(z))\subset B(z,\alpha r_n)\;,$$
where $\alpha>1$ is a constant which depends only on $f$ and $L$. Additionally, $f^n (B(z,r_n))$
contains a ball of radius comparable to $\delta$. 

Let $\lambda = 1/2$. 
Put  $\chi_i=1$ if there is $n$ such that $r_n\in A_i$, where $ A_i= [\lambda^{i+1}, \lambda^i)$,
$i$ is non-negative integer,  and $\chi_i=0$ otherwise. We want to show that there exists $\varepsilon>0$ such that
$$\liminf_{n\rightarrow \infty} \frac{\sum_{i=0}^{n-1}\chi_i}{n} \geq \varepsilon\;.$$

It is enough to prove that the sequence 
$[\log 1/r_{n}]$ is of positive density amongst integers, where $[\cdot]$ stands for an integer
part of a given number. Enumerate by $n_k$ all consecutive passages to the scale $\delta$ with bounded criticality $L$. By the exponential shrinking property, for every $x\in \mathcal{J}$
$\diam \mbox{Comp}_{f^{n_k-i}(x)} f^{-i}(B(f^{n_k}(x),\delta)$ is smaller than $\delta/2$ if only $i$ is large enough.   
Therefore,  without loss of generality, the modulus of the annulus
$(B(f^{n_{k}}(x), \delta) \setminus \mbox{Comp}_{f^{n_k}(x)}f^{-n_{k+1}+n_k}(B(f^{n_{k+1}}(x),\delta)$
is bigger than $\log 2$. By the  definition of $n_k$,  and Teichm\"uller's module theorem, there exists
a constant $P>0$ such that at  most $P$ consecutive numbers $\log \frac{1}{r_n}$ can be counted as the same  integer.  

Let $M = \sup_{z\in \cal J} |f'(z)|$. Then 
$$M^{-n}\leq \diam \mbox{Comp}_{z}f^{-n}(B_{\delta}(f^{n}(z))$$
and

\begin{eqnarray}
\frac{1}{N}\#(\{[\log \frac{1}{r_{k}}] : k\geq 0 \}\cap[0,N]) &\apprge&
\frac{1}{N P}\#\{ k\geq 0: 2k \kappa^{-1} \log M< N\}\nonumber \\
&\geq& \frac{\kappa}{2P\log M}:=\kappa_0\;\label{equ:kappa} 
\end{eqnarray}
provided $N$ is large enough. 

Suppose that the Julia set ${\cal J}$ of $f$ is not mean wiggly.
Hence,  for  every $\beta>0$ and $\rho\in (0,1)$  there is a point $z_\beta\in {\cal J}$ 
such that $B(z_\be, \lambda ^n)\cap {\cal J}$ is contained in 
a $\beta \lambda^n$ neighborhood of a line minimizing $\beta_{\cal J}(z_\beta,\lambda^n)$
and this property holds for the set $Z_\beta$ of integers $n\in \N$ of the density bigger than  $\rho$. 
Let us choose $\rho$ so that $\rho+\kappa_0>1$, where $\kappa_0$ is defined in the estimate (\ref{equ:kappa}).

For every $n\in Z_\beta$, 
\begin{description}
\item{\rm(i)} $f(B(z_\beta, \alpha r_n))\supset B_n=B(f^n(z),\delta')$\;\;\;\mbox{and}\;\;\;
$\delta'\sim \delta$, 
\item{\rm (ii)} the degree of $f^n$ on $B(z_\beta,r_n)$ is bounded by  $ L$. 
\end{description}

This means that there is  $\beta'>0$ (which does not depend on $n$) such  that the Julia set $\cal J$ in $B_n$
is contained in a $\beta'$-neighborhood of a finite union  of analytic Jordan arcs. Also, $\beta'$ tends
to $0$ when $\beta$ does so.
Passing to the limit with $\beta\rightarrow 0$, we obtain that  the Julia set in some ball $B$ of radius $\delta'$ 
is a finite union  of analytic Jordan  arcs.
By the eventually onto property ($B$ is mapped over ${\mathcal J}$ by an iterate of $f$), the whole $\cal J$ is a finite union  of analytic Jordan arcs. 

Since every rational function $f$ can have either $1$, $2$ or an infinite number of Fatou components, the piecewise analyticity
of $\cal J$ implies that $f$ has either $1$ or $2$ Fatou components. By connectivity of $\cal J$, every Fatou component
$\cal F$ is simply connected. Without loss of generality, we can assume that every
Fatou component is totally invariant by $f$. Also, $\cal J$ coincides with the boundary of every Fatou component. 
For every  $z\in \cal J$, we define the  set of angles of accesses $\{\theta_i(z)\}_{\cal F}$
from  inside of every  invariant Fatou component $\cal F$.  Observe that if  
$\{\theta_i(z)\}_{\cal F}\setminus \{\pi, 2\pi\}\not = \emptyset$  then  the same is true for every preimage  $y\in f^{-1}(z)$.
Since the set of preimages of a given $z\in \cal J$ is infinite and $\cal J$ is a union of finitely many Jordan analytic arcs, we infer  that for every $z\in \cal J$,   $\{\theta_{i}(z)\}_{\cal F} \subset\{\pi, 2\pi\}$. As a result,  $\cal J$ is either an analytic Jordan arc or an analytic circle. By the Fatou theorem~\cite{fat}, the Julia set coincides with a geometric circle or a segment.

\paragraph{Acknowledgement} The first author would like to thank G. David for many useful discussions about corona type constructions.

\end{document}